\documentclass[12pt]{article}
\usepackage{preprint-layout}
\usepackage{preprint-notation}

\title{A New Least Squares Stabilized Nitsche Method for Cut Isogeometric Analysis}

\author{Daniel Elfverson, Mats G. Larson and Karl Larsson}

\date{}

\newcommand{\supp}{\text{supp}}
\newcommand{\Span}{\text{span}}
\newcommand{\hatA}{\widehat{A}}

\begin{document}

\maketitle

\begin{abstract}
We derive a new stabilized symmetric Nitsche method for enforcement of Dirichlet boundary conditions for elliptic problems of second order in cut isogeometric analysis (CutIGA). We consider $C^1$ splines and stabilize the standard Nitsche method by adding a certain elementwise least squares terms in the vicinity of the Dirichlet boundary and an additional term on the boundary which involves the tangential gradient. We show coercivity with respect to the energy norm for functions in $H^2(\Omega)$ and optimal order a priori error estimates in the energy and $L^2$ norms. To obtain a well posed linear system of equations we combine our formulation with basis function removal which essentially eliminates basis functions with sufficiently small intersection with $\Omega$. The upshot of the formulation is that only elementwise stabilization is added in contrast to standard procedures based on ghost penalty and related techniques and that the stabilization is consistent. In our numerical experiments we see that the method works remarkably well in even extreme cut situations using a Nitsche parameter of moderate size.
\end{abstract}

\section{Introduction}
\paragraph{Earlier Work.} Cut finite element methods allow the geometric description of the 
computational domain to cut through the mesh in an arbitrary way. The resulting cut elements 
lead to difficulties on the Dirichlet boundary. Typically three approaches are used to handle 
this situation:
\begin{itemize}
\item Symmetric Nitsche in combination with stabilization using ghost penalties 
\cite{Bu10} and for the higher order case \cite{MaLaLoRo14} or element merging 
\cite{BaVe17,JoLa2013} which ensures that the 
necessary inverse inequality holds to guarantee coercivity. 
\item Symmetric Nitsche in combination with a sufficiently large value of the Nitsche parameter to ensure coercivity, see for instance \cite{PreVerBru17,PrLeMa17}.  This approach can 
also be combined with some stabilization for instance of finite cell type \cite{PaDu07} where 
a small amount of added stiffness is added to the full element.
\item Nonsymmetric Nitsche, see \cite{OdBa98}, which avoids the use of the inverse inequality 
to establish coercivity. Note however that some additional stabilization is necessary to establish 
a priori error estimates, see \cite{ElfLarLar18} for details.
\end{itemize}
The first alternative rests on a complete theoretical basis; the second is common in practice but 
optimal order a priori bounds can not be established in general since the penalty parameter 
may become very large, see the discussion in \cite{PrLeMa17}; and the third alternative was 
considered in \cite{ElfLarLar18} where a least squares term was added in the vicinity of 
the Dirichlet part of the boundary to provide the additional stability necessary to establish a priori 
error bounds. 

We refer to the overview article \cite{BuCl15}, and the recent conference proceedings 
\cite{BoBuLaOl17} for an overview of current research on cut element methods.

\paragraph{New Contributions.}
In this paper we develop a new symmetric Nitsche formulation for cut $C^1$ elements that is 
coercive but does not rely on ghost penalties or choosing a very large penalty parameter in 
the Nitsche penalty term. The method is instead based on adding two properly scaled consistent 
least squares terms. First one term on elements in the vicinity of the boundary and second 
one term which provides control of the tangent gradient along the boundary. The latter term is 
related to enforcement of the Dirichlet boundary condition in $H^{1/2}(\partial \Omega)$, 
where $\Omega$ is the computational domain with boundary $\partial \Omega$.
On standard elements which are not cut we may apply an inverse inequality and bound the least squares term involving the tangent gradient in terms of the standard Nitsche term but this is not possible on cut elements (unless some other stabilization is used). Therefore, on cut elements the stabilization term provides the additional control which reflects enforcement of the Dirichlet boundary condition in $H^{1/2}(\partial \Omega)$.

Together, these terms lead to a Nitsche formulation which is coercive with respect to the energy norm on $V = H^2(\Omega)$ in contrast to standard analysis of symmetric Nitsche which rely on inverse inequalities on the 
finite element space $V_h$. We utilize the added smoothness of the splines in our derivations 
and therefore only consider finite element spaces with at least $C^1$ regularity. The bulk 
part of the new least squares stabilization was used in \cite{ElfLarLar18} in combination with 
the non 
symmetric Nitsche formulation. Here we show that adding some additional stabilization also 
on the boundary leads to a natural extension of these results to the symmetric Nitsche formulation. 
Another interesting aspect of the new method is that we obtain explicit values of the stabilization 
parameters including the Nitsche parameter. 

We focus in particular on the B-spline spaces which are commonly used in isogeometric 
analysis \cite{CoHu09} but our approach is applicable to other types of $C^1$ finite element 
spaces, for instance tensor products of Hermite splines.

Our formulation is coercive on $H^2(\Omega)$ but the stiffness matrix is only guaranteed to 
be positive semidefinite since we allow elements with arbitrarily small intersection with the 
domain $\Omega$. In order to ensure that the stiffness matrix is positive definite we employ 
the recently introduced Basis Function Removal technique, see \cite{ElfLarLar18}, which builds 
on the obvious idea  of simply excluding basis functions with very small intersection. This 
can be done in such a way that accuracy is not lost. 

In our numerical investigations the method produces convergence results which are remarkably stable with respect to the cut situation. Also, the Nitsche penalty parameter in this case can be kept at a moderate size, which avoids problems due to locking when the boundary is curved within cut elements or boundary data is inhomogenous.

\paragraph{Outline.} The paper is organized as follows.
In Section 2 we present a model problem and introduce the least squares stabilized Nitsche formulation. In Section 3 we prove stability and error estimates in the energy and $L^2$ norms.
In Section 4 we detail the use of basis function removal to assure bounded condition numbers.
In Section 5 we present illustrating numerical examples which confirms the theoretical results. 

\section{The Model Problem and Method}

\subsection{The Dirichlet Problem}
Let $\Omega$ be a domain in $\IR^d$ with smooth boundary $\partial \Omega$ 
and consider the problem: find $u: \Omega \rightarrow \IR$ such that 
\begin{alignat}{3}
-\Delta u &=f & \qquad &\text{in $\Omega$}
\\
u &=g & \qquad &\text{on  $\partial \Omega$}
\end{alignat}
For sufficiently regular data there exists a unique solution to this problem.
Since we are interested in higher order methods we will always assume that the solution 
satisfies the regularity estimate 
 \begin{equation}\label{eq:regularity}
 \| u \|_{H^{s}(\Omega)} \lesssim 
 \|f \|_{H^{s-2}(\Omega)} 
 + 
 \|g \|_{H^{s-1/2}(\partial\Omega)}
 \end{equation}
for some $s \geq 2$. Here and below $a \lesssim b$ means that there is 
a positive constant $C$ such that $a \leq C b$.

\subsection{The B-Spline Spaces}

\begin{figure}
\centering
\includegraphics[width=0.9\linewidth]{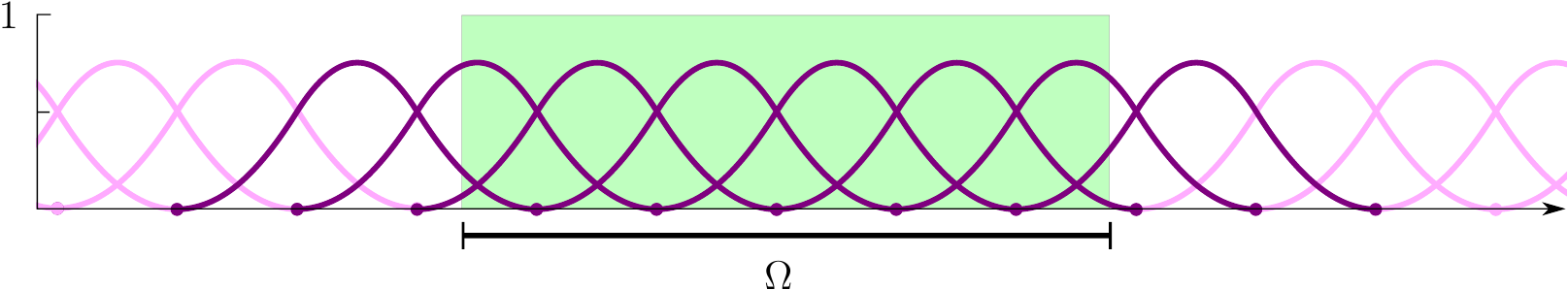}
\caption{Quadratic B-spline basis functions in one dimension. The set $B$ of basis functions with non-empty support in $\Omega$ are indicated in deep purple. Note that basis functions crossing the boundary 
of $\Omega$ are defined analogously to interior basis functions. Reproduced from \cite{ElfLarLar18} under the creative commons licence (\texttt{http://creativecommons.org/licenses/by/4.0/}).}
\label{fig:splines}
\end{figure}
\paragraph{Definitions.}
\begin{itemize}
\item Let $\widetilde{\mcT}_{h}$, $h \in (0,h_0]$, for some constant $h_0>0$, be a family of uniform tensor 
product meshes in $\IR^d$ with mesh parameter $h$. 
\item Let $\widetilde{V}_{h}=C^{p-1}Q^p(\IR^d)$ be the space of 
$C^{p-1}$ tensor product B-splines of degree $p$ defined on 
$\widetilde{\mcT}_{h}$. Let $\widetilde{B} = \{\varphi_i\}_{ i\in \widetilde{I}}$ be 
the standard basis in $\widetilde{V}_h$, where $\widetilde{I}$ is an index set.
\item Let $B = \{ \varphi \in \widetilde{B} \, : \,
\supp(\varphi) \cap \Omega \neq \emptyset \}$ be the set of basis functions with 
support that intersects $\Omega$. Let $I$ be an index set for $B$. Let $V_h = \Span\{B\}$ 
and let $\mcT_h = \{ T \in \widetilde{\mcT}_h : T \subset \cup_{\varphi \in B } \supp(\varphi) \}$. 
\item We will only consider $p \geq 2$ corresponding to at least $C^1$ splines. We 
then have $V_h \subset V = H^2(\Omega)$. 
The case $p=2$ in 1D is illustrated in Figure~\ref{fig:splines}.
\end{itemize}
\begin{rem} To construct the basis functions in $\widetilde{V}_h$ 
we start with the one dimensional line $\IR$ and define a uniform partition, 
with nodes $x_i = i h$, $i \in \mathbb{Z}$, where $h$ is the mesh parameter, 
and elements $I_i = [x_{i-1},x_i)$. We define 
\begin{equation}
\varphi_{i,0}(x) 
=
\begin{cases}
1 & x \in I_i
\\
0 & x \in \IR \setminus I_i
\end{cases}
\end{equation}
The basis functions $\varphi_{i,p}$ are then defined by the Cox-de Boor recursion 
formula 
\begin{equation}
\varphi_{i,p} = \frac{x - x_i}{x_{i+p} - x_i} \varphi_{i,p-1}(x) 
+ \frac{x_{i+p+1} - x}{x_{i+p+1} - x_{i+1}} \varphi_{i+1,p-1}(x) 
\end{equation}
we note that these basis functions are $C^{p-1}$ and supported on 
$[x_i,x_{i+p+1}]$ which corresponds to $p+1$ elements.
We then define tensor product basis functions in 
$\IR^d$ of the form 
\begin{equation}
\varphi_{i_1,\dots,i_d} (x) = \prod_{k=1}^d \varphi_{i_k}(x_k)
\end{equation}
\end{rem}

\subsection{The Finite Element Method}

\paragraph{Method.}
Find 
$u_{h} \in V_{h}$ such that 
\begin{equation}\label{eq:method}
A_{h}(u_{h},v) = L_h(v)\qquad v \in V_{h}
\end{equation}

\paragraph{Forms.} Let $\delta \sim h$ be a parameter and define the forms
\begin{align}\label{eq:Ah}
A_h(v,w) &= a_h(v,w) 
- (n\cdot \nabla v,w)_{\partial \Omega} 
- (v, n\cdot \nabla w)_{\partial \Omega} 
+\beta b_h(v,w)
\\ \label{eq:ah}
a_h(v,w)&= (\nabla v, \nabla w)_\Omega
+ \tau \delta^2(\Delta v,\Delta w)_{\mcT_{h,\delta} \cap \Omega}
\\ \label{eq:bh}
b_h(v,w)&= 
  (2 + \tau^{-1}) \delta^{-1} (v,w)_{\partial \Omega} 
+ 2\delta (\nabla_T v, \nabla_T w)_{\partial \Omega} 
\\ \label{eq:Lh}
L_h(v) &= (f, v)_\Omega
- \tau \delta^2 (f,\Delta v)_{\mcT_{h,\delta}\cap \Omega}
- (g,n\cdot \nabla v)_{\partial\Omega} + \beta b_h(g,v)
\end{align}
Here we used the notation:
\begin{itemize}
\item $\beta$ is the penalty parameter which can take a moderate value for instance $5 \leq \beta$ 
and $\tau$ is a positive parameter which enables us to trade weight between the least squares 
bulk term and the standard Nitsche term. We note in practice that a small $\tau$ leads to a more accurate method. We refer to the numerical section for more details on the choice of parameters.
\item $\nabla_T$ is the tangential gradient at $\partial \Omega$ defined by $\nabla_T= P \nabla$, where $P = I - n\otimes n$ is the projection of vectors in $\IR^d$ onto the tangent plane of the boundary $\partial \Omega$. 
\item In \eqref{eq:ah} we used the form
\begin{equation}
(v,w)_{\mcT_{h,\delta} \cap \Omega} 
= \sum_{T \in \mcT_{h,\delta}} (v,w)_{T\cap \Omega}
\end{equation}
where $\mcT_{h,\delta} \subset \mcT_h$ is defined by 
\begin{equation}\label{def:Th}
\mcT_{h,\delta} = \mcT_h(U_\delta (\partial \Omega))
= \{T \in \mcT_h : T \cap U_\delta (\partial \Omega) \neq \emptyset \}  
\end{equation}
and
\begin{equation}\label{eq:Udelta}
U_\delta (\partial \Omega) = \left(\bigcup_{x\in \partial \Omega} B_\delta(x)\right) 
\cap \Omega
\end{equation}
with $\delta \sim h$ and $B_\delta(x)$ the open ball with center $x$ and 
radius $\delta$. See Figure~\ref{fig:spline-bean} for an illustration.
\end{itemize}

\begin{figure}
\centering
\begin{subfigure}[t]{.3\linewidth}
\includegraphics[width=\linewidth]{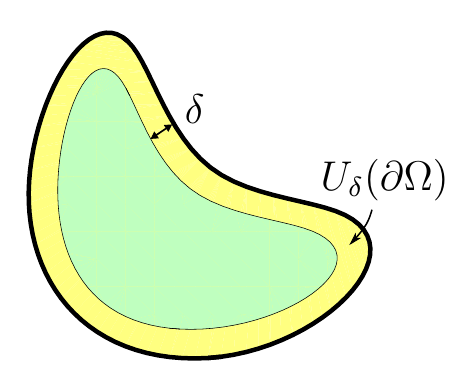}
\end{subfigure}
\quad
\begin{subfigure}[t]{.3\linewidth}
\includegraphics[width=\linewidth]{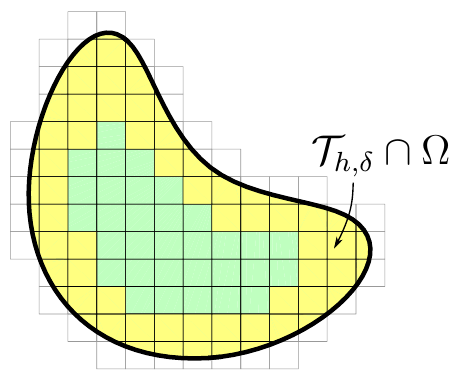}
\end{subfigure}
\caption{Illustrations of subdomains $U_\delta(\partial\Omega)$ and $\mcT_{h,\delta}\cap\Omega$ where $\mcT_{h,\delta}$ is extracted according to Remark~\ref{rem:Th-practice}.}
\label{fig:spline-bean}
\end{figure}

\paragraph{Galerkin Orthogonality.} It holds 
\begin{equation}\label{eq:galort}
A_h( u - u_h, v) = 0 \qquad \forall v \in V_h
\end{equation}
This identity follows directly from the consistency of the standard Nitsche method and the fact that 
we have only added consistent least squares terms.

\begin{rem} \label{rem:Th-practice}
In practice, if $\delta= h$ is used $\mcT_{h,\delta}$ may be taken as the set of all elements that intersect 
the Dirichlet boundary $\partial \Omega$ and their neighbors, i.e. $\mcT_{h,\delta} = 
\mcN_h(\mcT_h(\partial \Omega))$, see Figure~\ref{fig:spline-bean}.
\end{rem}

\begin{rem} We note that in addition to the usual Nitsche formulation we have an interior least 
squares term and also a term providing control of the tangent derivative along the boundary. Both 
of these terms can be computed elementwise using the standard assembly of the stiffness matrix. 
Note also that the stabilization terms do not cause fill in which is the case with standard ghost 
penalty approaches \cite{Bu10, JoLaLa17, MaLaLoRo14}.
\end{rem}

\section{Error Estimates}
\subsection{Properties of $\boldsymbol{A}_{\boldsymbol{h}}$}
\paragraph{Norms.}
Define the norms 
\begin{align}\label{eq:energy-norm}
 \tn v \tn_{h}^2 &= \| v \|^2_{a_h} + \| v \|^2_{b_h}
 \\
\| v \|^2_{a_h} &= a_h(v,v) = \| \nabla v \|^2_\Omega 
+ \tau \delta^2 \| \Delta v \|^2_{\mcT_{h,\delta} \cap \Omega}
\\
\| v \|^2_{b_h} &= b_h(v,v) = (2 + \tau^{-1})\delta^{-1} \| v \|^2_{\partial \Omega} + 2 \delta \| \nabla_T v \|^2_{\partial \Omega}
\end{align}

\paragraph{Technical Lemmas.}
Let $\rho$, with $\rho>0$ in $\Omega$, be the signed 
distance function associated with $\partial \Omega$ and recall that the closest point mapping 
$\eta : U_\delta \rightarrow \partial \Omega$ is well defined for $0<\delta \leq \delta_0$ 
where 
\begin{equation}\label{eq:delta-0-bound}
\delta_0 \| \kappa \|_{L^\infty(\partial \Omega)} \leq C_1 < 1
\end{equation}
Here $\kappa = \nabla^2 \rho$ is the curvature tensor of $\partial \Omega$ and 
$\| \kappa \|_{L^\infty(\partial \Omega)} = \sup_{x \in \partial \Omega} \| \kappa \|_{\IR^d}$ where $\| \cdot \|_{\IR^d}$ on a $d\times d$ matrix is the operator norm. 
See \cite{GiTr01} for further details.
Through the closest point mapping we extend functions $v$ defined on $\partial\Omega$ onto $U_\delta$ and we use the notation
\begin{equation}
v^e = v \circ \eta
\end{equation}
Furthermore, we will let $C$ denote constants 
independent of the mesh, the domain, and the intersection between the domain and 
the mesh, which are not the same at all occurances. An index is used to indicate the constant 
in a specific inequality.

We consider mesh parameters $h \in (0,h_0]$, and let $\delta$ 
be a parameter such that 
\begin{equation}\label{eq:assum-h-delta}
0< h \leq \delta \leq h_0 \leq \delta_0
\end{equation}
which essentially means that that the mesh resolves the boundary.

\begin{lem} \label{lem:aux1}
Let $\delta$ satisfy (\ref{eq:delta-0-bound}), then
\begin{align}\label{eq:ext-bound-l2}
\| v^e \|^2_{U_\delta} &\lesssim \delta \| v \|^2_{\partial \Omega}
\\ \label{eq:ext-bound-nabla}
\| \nabla v^e \|^2_{U_\delta} &\lesssim \delta \| \nabla_T v \|^2_{\partial \Omega}
\end{align}
where the hidden constant takes the form $1+ C \delta \| \kappa \|_{L^\infty(\partial \Omega)}$.
\end{lem}
\begin{proof}
For $t \in (0,\delta_0)$ we define  $\partial \Omega_t = \{ x \in \Omega : \rho(x) = t \}$. Let 
$dx_t$ be the surface measure on $\partial \Omega_t$. Then we have $dx_t = \mu_t dx$, 
where $dx$ is the surface measure on $\partial \Omega$, and the following estimate 
\begin{equation}\label{eq:measure}
\|\mu_t\|_{L^\infty(\partial \Omega)} \leq 1 +  C t \| \kappa \|_{L^\infty(\partial \Omega)}
\end{equation}

\paragraph{(\ref{eq:ext-bound-l2})} Using (\ref{eq:measure}) we obtain 
\begin{align}
\| w^e \|^2_{U_\delta} 
&=
\int_0^\delta \| w^e \|^2_{\partial \Omega_t} \, dt 
\\
&=
\int_0^\delta \int_{\partial \Omega_t} |w^e|^2 \, d x_t \, dt 
\\
&= 
\int_0^\delta \int_{\partial \Omega} |w|^2 \mu_t \, d x \, dt
\\
&
\leq \int_0^\delta \left( 1 + C t   \| \kappa \|_{L^\infty(\partial \Omega)} \right) \, dt  \int_{\partial \Omega} |w|^2 \, d x
\\
&
= \delta \left(1 + C \frac{\delta}{2}  \| \kappa \|_{L^\infty(\partial \Omega)} \right)
\| w \|^2_{\partial \Omega}
\end{align}

\paragraph{(\ref{eq:ext-bound-nabla}).} Using the chain rule we have 
\begin{equation}\label{eq:chain-rule}
(\nabla v^e)|_x  = (D\eta|_x)^T  (\nabla v)|_{\eta(x)} 
= (I - \rho(x) \kappa(x))^T (P\nabla v)|_{\eta(x)}
\end{equation}
where $D \eta (x)  = I - n(\eta(x)) \otimes n(\eta(x)) - \rho(x) \kappa(x)$ 
is the Jacobian of the closest point mapping. Thus we obtain the estimate 
\begin{equation}
\| \nabla v^e(x) \|_{\IR^d} 
\leq \| I - \rho(x) \kappa(x) \|_{\IR^d} \| (\nabla_T v)^e \|_{\IR^d} 
\leq \left( 1 + \rho(x) \| \kappa(x) \|_{\IR^d} \right) \| (\nabla_T v)^e \|_{\IR^d}
\end{equation}
where $\| \kappa(x) \|_{\IR^d} \leq (1 - C_1)^{-1} \| \kappa (\eta (x )) \|_{\IR^d} 
\leq C \| \kappa( \eta (x )) \|_{\IR^d}$, and $C_1$ is the constant in (\ref{eq:delta-0-bound}), 
which we may take to $1/2$ for instance. We thus have 
\begin{equation}
\| \nabla v^e (x)  \|_{\IR^d} \leq \big(1 + C \rho(x) \| \kappa \|_{L^\infty(\partial \Omega)} \big)  \| (\nabla_T v)^e \|_{\IR^d} \qquad x \in U_\delta
\end{equation}
We now estimate \eqref{eq:ext-bound-nabla} in as follows
\begin{align}
\| \nabla w^e \|^2_{U_\delta} 
&=
\int_0^\delta \| \nabla w^e \|^2_{\partial \Omega_t} \, dt 
\\
&=
\int_0^\delta \int_{\partial \Omega_t} \| \nabla w^e\|_{\IR^d}^2 \, d x_t \, dt 
\\
&=
\int_0^\delta \int_{\partial \Omega_t} \| (I - t \kappa) (\nabla_T w)^e\|_{\IR^d}^2 \, d x_t \,dt 
\\
&\leq 
\int_0^\delta 
\left(1 + C t \| \kappa \|_{L^\infty(\partial \Omega)} \right)^2 \|\mu_t\|_{L^\infty(\partial \Omega)} \,dt 
\int_{\partial \Omega} \|\nabla_T v\|_{\IR^d}^2  \,dx 
\\
&
\lesssim 
\delta \bigl( 1 + C \delta \| \kappa \|_{L^\infty(\partial \Omega)} \bigr)
 \| \nabla_T w \|^2_{\partial \Omega}
\end{align}
Here we used the estimate 
\begin{align}
\nonumber
&\int_0^\delta 
\bigl(1 + C t \| \kappa \|_{L^\infty(\partial \Omega)} \bigr)^2 \|\mu_t\|_{L^\infty(\partial \Omega)} \, dt 
\\
&\qquad\quad
\leq
\int_0^\delta 1 + 3 C t \| \kappa \|_{L^\infty(\partial \Omega)} 
+ 3(C t)^2 \| \kappa  \|^2_{L^\infty(\partial \Omega)}
+ (C t)^3 \| \kappa  \|^3_{L^\infty(\partial \Omega)}
\, dt
\\
&\qquad\quad
= 
\delta \biggl(1 + 
\frac{3}{2} C\delta \| \kappa \|_{L^\infty(\partial \Omega)} 
+ (C \delta)^2 \| \kappa \|^2_{L^\infty(\partial \Omega)}
+ \frac{1}{4}(C t)^3 \| \kappa  \|^3_{L^\infty(\partial \Omega)}
\biggr)
\\
& \qquad\quad
\lesssim 
\delta \bigl( 1 + C \delta \| \kappa \|_{L^\infty(\partial \Omega)} \bigr)
\end{align}
where we at last used (\ref{eq:delta-0-bound}). 
\end{proof}

\begin{lem}[Continuity for the Consistency Form] \label{lem:const-bound}
The following estimate holds 
\begin{align}\label{eq:trace}
\left| (n\cdot \nabla v, w)_{\partial \Omega} \right|
&\lesssim 
\| v \|_{a_h} \| w \|_{b_h}\qquad v,w \in V
\end{align}
with a hidden constant of the form $( 1 + C\delta \| \kappa\|_{L^\infty(\partial\Omega)})^{1/2}$.
\end{lem}
\begin{proof} Given $\delta>0$ let $\chi:\Omega \rightarrow [0,1]$ be defined by
\begin{equation}
\chi(x)
=
\begin{cases}
1 - \rho(x)/\delta & \qquad \text{$x \in U_\delta$}
\\
0  & \qquad \text{$x \in \Omega \setminus U_\delta$}
\end{cases}
\end{equation}
Then we have 
\begin{equation}
\| \chi \|_{L^\infty(\Omega)} = 1, \qquad \| \nabla \chi \|_{L^\infty(\Omega)} = 1/\delta
\end{equation}
Using Green's formula we have the identity 
\begin{align}
(\nabla v, \chi \nabla w^e )_{U_\delta} 
&=(n\cdot \nabla v, \chi w^e)_{\partial \Omega} 
-(\Delta v, \chi w^e)_{U_\delta} - (\nabla v, (\nabla \chi) w^e)_{U_\delta}
 \end{align}
By rearranging the terms, applying the triangle inequality and Cauchy--Schwarz inequality we have
\begin{align}
&\left| (n\cdot \nabla v, \chi w^e)_{\partial \Omega} \right|
\leq
\left| (\nabla v, \chi \nabla w^e)_{U_\delta} \right|
+ 
\left| (\Delta v, \chi w^e)_{U_\delta} \right|
+ 
\left| (\nabla v, (\nabla \chi) w^e)_{U_\delta} \right|
\\
&\qquad \leq
\|\nabla v\|_{U_\delta} \|\nabla w^e \|_{U_\delta}  
+ 
\tau_*^{1/2} \delta \|\Delta v\|_{U_\delta}  \tau_*^{-1/2} \delta^{-1} \|w^e \|_{U_\delta} 
+ 
\|\nabla v\|_{U_\delta} \delta^{-1} \|w^e \|_{U_\delta}
\\
&\qquad \leq
\Big( 2 \|\nabla v\|^2_{U_\delta} 
+ 
\tau_* \delta^2 \|\Delta v\|^2_{U_\delta} \Big)^{1/2} 
\Big(( 1 + \tau_*^{-1} ) \delta^{-2}\|w^e \|^2_{U_\delta} 
+ 
\|\nabla w^e \|^2_{U_\delta} \Big)^{1/2} 
\\
&\qquad \lesssim 
\Big( 2 \|\nabla v\|^2_{U_\delta} 
+ 
\tau_* \delta^2 \|\Delta v\|^2_{U_\delta} \Big)^{1/2} 
\Big(( 1 + \tau_*^{-1} ) \delta^{-1}\|w \|^2_{\partial \Omega} 
+ 
\delta \|\nabla_T w \|^2_{\partial \Omega} \Big)^{1/2} 
\end{align}
where we in the last inequality use Lemma~\ref{lem:aux1}.
Choosing $\tau_* = 2\tau$ we finally arrive at  
\begin{align}
\left| (n\cdot \nabla v, w)_{\partial \Omega} \right|
&\lesssim 
\left(\|\nabla v\|^2_{U_\delta} 
+ 
\tau \delta^2 \|\Delta v\|^2_{U_\delta} \right)^{1/2}
\\&\quad \ \nonumber
\times \left( (2 + \tau^{-1}) \delta^{-1} \| w \|^2_{\partial \Omega} 
+ 
2 \delta \| \nabla_T w \|^2_{\partial \Omega} \right)^{1/2}
\end{align}
where the hidden constant takes the form $\bigl( 1 + C\delta \| \kappa\|_{L^\infty(\partial\Omega)} \bigr)^{1/2}$.
Taking $\delta \sim h$ gives the desired estimate.
\end{proof}

\begin{lem}[Coercivity] For $\beta>0$ sufficiently large the form 
$A_h$ is coercive 
\begin{equation}\label{eq:coercivity}
\tn v \tn_{h}^2 \lesssim A_h(v,v) \qquad v \in V
\end{equation}
where $V = H^2(\Omega)$.  
\end{lem}
\begin{rem} Note that the coercivity with respect to the energy norm holds on the full space $V = H^2(\Omega)$ and not only 
on the discrete space $V_h$. The reason is of course that we do not invoke any inverse inequality 
in the proof of the continuity for the consistency form (Lemma~\ref{lem:const-bound}).
\end{rem}
\begin{proof} By standard calculations, the continuity for the consistency form (Lemma~\ref{lem:const-bound}) and a Young's inequality ($2ab \leq \gamma a^2 + \gamma^{-1}b^2$ for any $a,b,\gamma>0$) we have 
\begin{align}
A_h(v,v) &=\| v \|^2_{a_h}
- 2(n\cdot \nabla v, v)_{\partial \Omega} 
+\beta \| v \|^2_{b_h}
\\
&\geq
\| v \|^2_{a_h}
- 2 \left| (n\cdot \nabla v, v)_{\partial \Omega} \right|
+\beta \| v \|^2_{b_h}
\\
&\geq 
\| v \|^2_{a_h} 
- 2 c \| v \|_{a_h} \| v \|_{b_h}
+\beta \| v \|^2_{b_h}
\\
&\geq
(1 - \gamma c ) \| v \|^2_{a_h}
+ ( \beta  - \gamma^{-1} c) \| v \|^2_{b_h}
\end{align}
where $c$ is the hidden constant in the continuity for the consistency form (\ref{eq:trace}).
\end{proof}

\begin{rem}  The constant $c$ is of very moderate size in fact it is close to one 
and approaches one as the mesh size tends to zero. Therefore the constant $\beta$ 
may be chosen to have a moderate value, for instance we may take $5 \leq \beta$.
\end{rem}

\begin{lem}[Continuity] The form $A_h$ is continuous 
\begin{equation}\label{eq:continuity}
\left| A_h(v,w) \right| \lesssim \tn v \tn_h
\tn w \tn_{h} \qquad v,w \in V 
\end{equation}
\end{lem}
\begin{proof} We have
\begin{align}
\left| A_h(v,w) \right| &\leq  \left| a_h(v,w) \right|
+ \left| (n\cdot \nabla v, w)_{\partial \Omega} \right|
+ \left| (v, n \cdot \nabla w )_{\partial \Omega} \right|
+ \beta \left| b_h(v,w) \right|
\\
&\lesssim \| v\|_{a_h} \| w\|_{a_h}  
+ \|v\|_{a_h} \|w\|_{b_h}  
+ \|v\|_{b_h} \|w\|_{a_h}
+ \beta \|v\|_{b_h} \|w\|_{b_h}
\\
&\leq (2\|v\|^2_{a_h} + (1+\beta)\| v \|^2_{b_h} )^{1/2} 
( 2 \|w\|^2_{a_h} + (1+\beta)\| w \|^2_{b_h} )^{1/2} 
\\
&\leq
\max\left(2,1+\beta\right)
\tn v \tn_h \tn w \tn_h
\end{align}
where we used Lemma~\ref{lem:const-bound} followed by the 
Cauchy-Schwarz inequality.
\end{proof}

\subsection{Interpolation Error Estimates}
\label{section:interpolation}

There is an extension operator $E:W^k_q(\Omega) \rightarrow W^k_q(\IR^d)$, $k\geq 0$ 
and $q\geq 1$, such that 
\begin{equation}\label{eq:extension-continuity}
\| E v \|_{W^k_q(\IR^d)} \lesssim \| v \|_{W^k_q(\Omega)}
\end{equation}
see \cite{Fo95}. Define the interpolant by 
\begin{equation}\label{eq:interpolant-def}
\pi_{h} : H^s(\Omega) \ni u \mapsto \pi_{Cl,h} ( E u ) \in V_h  
\end{equation} 
where $\pi_{Cl,h}$ is a Clement type interpolation operator 
onto the spline space $V_h$.
By including the extension operator $E$ in the definition of the interpolation operator \eqref{eq:interpolant-def} we can utilize interpolation results on full elements, also in the case when an element is cut by the boundary.
We have the standard elementwise a priori error 
estimate 
\begin{equation}\label{eq:interpol-basic}
\| v - \pi_h v \|_{H^m(T)} \lesssim h^{s - m } \| v \|_{H^s(\mcN_h(T))}
\end{equation}
where $\mcN_h(T) = \cup_{\varphi \in B(T)} \supp(\varphi)$ and 
$B(T) = \{ \varphi \in B : T \subset \supp(\varphi)\}$.  See \cite{IGA-h-refined} for interpolation 
results for spline spaces. 
\begin{lem}
We have the interpolation estimate 
\begin{equation}\label{eq:interpol-energy}
\tn v - \pi_h v \tn_h \lesssim h^p \| v \|_{H^{p+1}(\Omega)}
\end{equation}
\end{lem}
\begin{proof} Where needed, let the extension operator be implied, i.e., $v=Ev$.
We directly have the estimates 
\begin{align}
\|\nabla ( v - \pi_h v )\|^2_{\Omega} &\lesssim h^{2p} \| v \|^2_{H^{p+1}(\Omega)}
\\
h^2 \|\Delta (v - \pi_h v )\|^2_\Omega &\lesssim h^{2p} \| v \|^2_{H^{p+1}(\Omega)}
\end{align}
For the boundary terms we employ the trace 
inequality 
\begin{equation}
\| v \|^2_{\partial \Omega} \lesssim 
h^{-1} \|v \|^2_{\mcT_h(\partial \Omega)} + h \| \nabla v \|^2_{\mcT_h(\partial \Omega)} 
\end{equation}
see \cite{HaHaLa03}, which together with the interpolation 
bound (\ref{eq:interpol-basic}) and the stability of the extension operator (\ref{eq:extension-continuity}) give
\begin{align}
h^{-1} \| v - \pi_h v  \|^2_{\partial \Omega} &\lesssim h^{2} \| v - \pi_h v  \|^2_{\mcT_h(\partial \Omega)} 
+ \| \nabla (v  - \pi_h v) \|^2_{\mcT_h(\partial \Omega)} 
\\
&\lesssim h^{2p} \| v \|^2_{H^{p+1}(\mcN_{h}(\mcT_h(\partial \Omega) )}
\\
&\lesssim h^{2p} \| v \|^2_{H^{p+1}(\Omega)}
\\
h \| \nabla_T (v - \pi_h v ) \|^2_{\partial \Omega} 
&\lesssim  \| \nabla (v - \pi_h v)  \|^2_{\mcT_h(\partial \Omega)} 
+ h^2 \| \nabla^2 (v  - \pi_h v) \|^2_{\mcT_h(\partial \Omega)} 
\\
&\lesssim h^{2p} \| v \|^2_{H^{p+1}(\mcN_{h}(\mcT_h(\partial \Omega) )}
\\
&\lesssim h^{2p} \| v \|^2_{H^{p+1}(\Omega)}
\end{align}
Summing these four bounds we directly obtain (\ref{eq:interpol-energy}).
\end{proof}
\subsection{Error Estimates}
\begin{thm}
The following error estimates holds
\begin{align}\label{eq:errorest-energy}
\tn u - u_h \tn_h &\lesssim h^p \| u \|_{H^{p+1}(\Omega)}
\\
\label{eq:errorest-L2}
\| u - u_h \|_{\Omega} &\lesssim h^{p+1}\| u \|_{H^{p+1}(\Omega)}
\end{align}
\end{thm}
\begin{proof} Estimate (\ref{eq:errorest-energy}) follows directly from 
the coercivity, the Galerkin orthogonality (\ref{eq:galort}), the continuity 
and finally the interpolation error estimate.  Estimate (\ref{eq:errorest-L2}) 
follows from (\ref{eq:errorest-energy}) using a standard duality argument.
\end{proof}

\section{Conditioning of the Discrete System}
When constructing a robust fictitious domain method we must take into account the following two stability issues that are induced in some cut situations:
\begin{itemize}
\item
\emph{Loss of Coercivity.} Coercivity is a most fundamental property of PDEs in weak form, which, for instance, is used to prove existence and uniqueness of solutions to elliptic PDEs.
When proving coercivity there is some flexibility in the choice of norm to work in
and notably our present choice, the energy norm \eqref{eq:energy-norm}, is closely related to the continuous problem and hence includes no information outside of the domain.

\item
\emph{Unbounded Condition Number.}
When proving a bound on the condition number a norm on the coefficients in the approximation space is used. In most cut situations this norm includes degrees of freedoms associated with nodes outside of the domain and for that reason this norm cannot be bounded by the energy norm alone.
Therefore, as shown in the next section, guaranteed coercivity in the energy norm does not imply a bound on the condition number.

\end{itemize}
Often, the remedy of both these symptoms are treated using the same medicine, for example by adding ghost penalty \cite{Bu10} or a small amount of stiffness in cut elements outside of the domain \cite{PaDu07}. In the present work we take a different approach, treating the symptoms separately, where the coercivity is guaranteed by adding suitable terms to the method, whereas the condition number is ensured to be bounded via basis removal as described next.

\subsection{Basis Function Removal} \label{section:basis-removal}

The stiffness matrix $\hatA$ is defined by
\begin{equation}\label{eq:stiffnessmatrix}
(\hatA \hatv,\hatw)_{\IR^N} = A_h(v,w)\qquad \forall v,w \in V_h
\end{equation}
where $\hatv \in \IR^N$ is the coefficient vector in the expansion 
\begin{equation}
v(x) = \sum_{i\in I} \hatv_i \varphi_i(x)
\end{equation}
of $v$ in terms of the basis functions. We note that $\hatA$ is symmetric 
and positive semidefinite since
\begin{equation}
(\hatA \hatv,\hatv)_{\IR^N} = A_h(v,v) \gtrsim \tn v \tn^2_h \geq 0 
\end{equation}
Note however that $\tn v \tn_h$ is only a semi norm on $V_h$ since the functions 
in $V_h$ are defined on $\Omega_h = \cup_{i \in I} \supp(\varphi_i)$ and 
$\Omega \subset \Omega_h$ and there may be basis functions $\varphi \in B$ 
such that the intersection $\supp(\varphi) \cap \Omega$ is arbitrarily small. 

To obtain a positive definite stiffness matrix we apply basis function removal.
This approach was analyzed in \cite{ElfLarLar18} and builds 
on the the idea of systematically removing basis functions that have sufficiently small 
intersection with the domain. This is done in such a way that optimal order accuracy in 
a specified norm is retained. For instance, using the energy norm we may remove basis functions according to the following procedure.
Assuming the basis functions in $B = \{ \varphi_i \}_{i=1}^N$ are sorted such that the size of their energy norms $\tn \varphi_i \tn_h$ are ascending we may remove the $N_r$ first basis functions as long as
\begin{equation}
\sum_{i=1}^{N_r} \tn \varphi_i \tn_h^2 \leq tol^2
\end{equation}
with $tol = c h^{p}$. See Figure~\ref{fig:basis-removal} for an example selection of basis functions.
This natural procedure leads to optimal order convergence 
and a stiffness matrix which is uniformly symmetric positive definite independent of 
the position of the domain in the background mesh. We refer to \cite{ElfLarLar18} for 
further details.  The resulting linear system of equations may then be solved using a direct 
solver or we may apply an iterative solver combined with a preconditioner. We refer to 
\cite{PreVerBru17} and \cite{PreVerZwe17} for preconditioning of cut element methods.

\begin{figure}
\centering
\begin{subfigure}[t]{.3\linewidth}
\includegraphics[width=\linewidth]{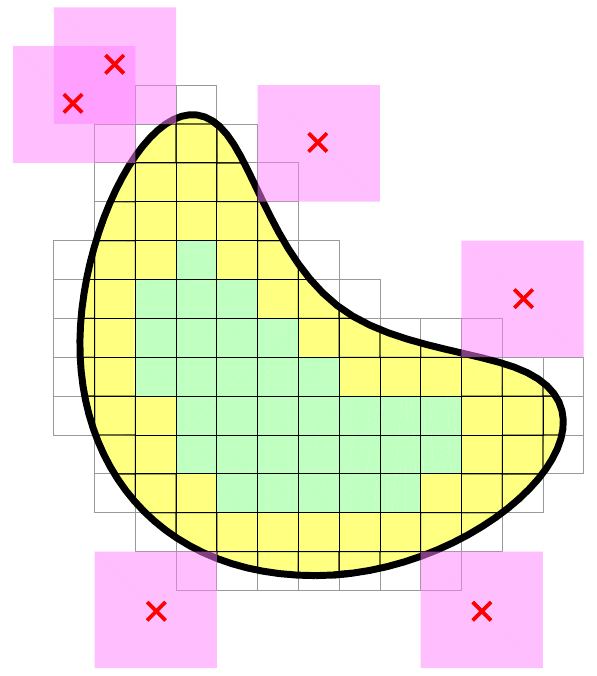}
\end{subfigure}
\caption{Illustration of quadratic B-spline basis functions selected for basis removal. For each selected basis function the location of its peak is indicated by a red cross and its support is indicated by a pink square.}
\label{fig:basis-removal}
\end{figure}

\section{Numerical Results}

\paragraph{Implementation.}
The new least squares stabilized Nitsche method in 2D is implemented in MATLAB and the linear system of equations is solved using a direct solver (MATLAB's \verb+\+ operator).
We use tensor product quadratic B-spline basis functions, i.e. $C^1$-splines, in all experiments. The domain $\Omega$ is described as a high resolution polygon. The additional terms, with regards to the standard symmetric Nitsche method which we formulate below, are added elementwise.

\paragraph{Parameter Values.}
In our experiments we let $\delta=h$ and extract the subdomain $\mcT_{h,\delta} \cap \Omega$ according to Remark~\ref{rem:Th-practice}. We use $\beta=10$ while the value of $\tau>0$ is varied. When basis removal is active we use a tolerance constant $c=0.01$ as described in Section~\ref{section:basis-removal}.

\paragraph{The Standard Symmetric Nitsche Method.} For comparison we also include results for the standard symmetric Nitsche method defined via the forms
\begin{align}
A_{h,\mathrm{std}}(v,w) &= (\nabla v, \nabla w)_\Omega
-(n\cdot\nabla v, w)_{\partial\Omega} -(v, n\cdot\nabla w)_{\partial\Omega}
\\&\qquad + h^{-1}(\beta(2+\tau^{-1})v,w)_{\partial\Omega} \nonumber
\\
L_{h,\mathrm{std}}(v) &= (f,v)_\Omega - (g, n\cdot\nabla v)_{\partial\Omega} + h^{-1}(\beta(2+\tau^{-1})g,v)_{\partial\Omega}
\end{align}
To make for an easier comparison the Nitsche penalty parameter is written $\beta(2+\tau^{-1})$ just as in our formulation of the least squares stabilized method \eqref{eq:method}.

\begin{rem} \label{rem:std-nitsche}
In the results below we in no way ensure coercivity of the standard symmetric Nitsche method. That could be performed using the techniques outlined in the introduction, for example by suitably increasing the penalty parameter. Thus, in examples below where the standard Nitsche method appears to perform poorly one could argue that this is a consequence of poor parameter choices for this method. However, we still found it interesting to include this comparison.
\end{rem}

\subsection{Model Problems}

We manufacture both our model problems based on the ansatz
\begin{align}
u(x,y) = \frac{1}{10}(\sin(2x) + x\cos(3y))
\end{align}
from which we, given a domain $\Omega$, derive the input data $f$ in $\Omega$ and $g,\nabla_T g$ on $\partial\Omega$.

\paragraph{Unit Square.} The first geometry we consider is the unit square $\Omega = [0,1]^2$ on which we examine both fitted meshes and cut meshes generated in a controlled fashion. The cut situations are created by letting the topmost and rightmost elements extend outside the domain a distance $h\delta_\mathrm{cut}$. This is further described and illustrated in Figure~\ref{fig:unit-square-meshes}. An example numerical solution to this model problem is shown in Figure~\ref{fig:unit-square-solution}. Note that we do not activate basis removal in the examples based on this model problem.

\begin{figure}
\centering
\begin{subfigure}[t]{.26\linewidth}\centering
\includegraphics[width=\linewidth]{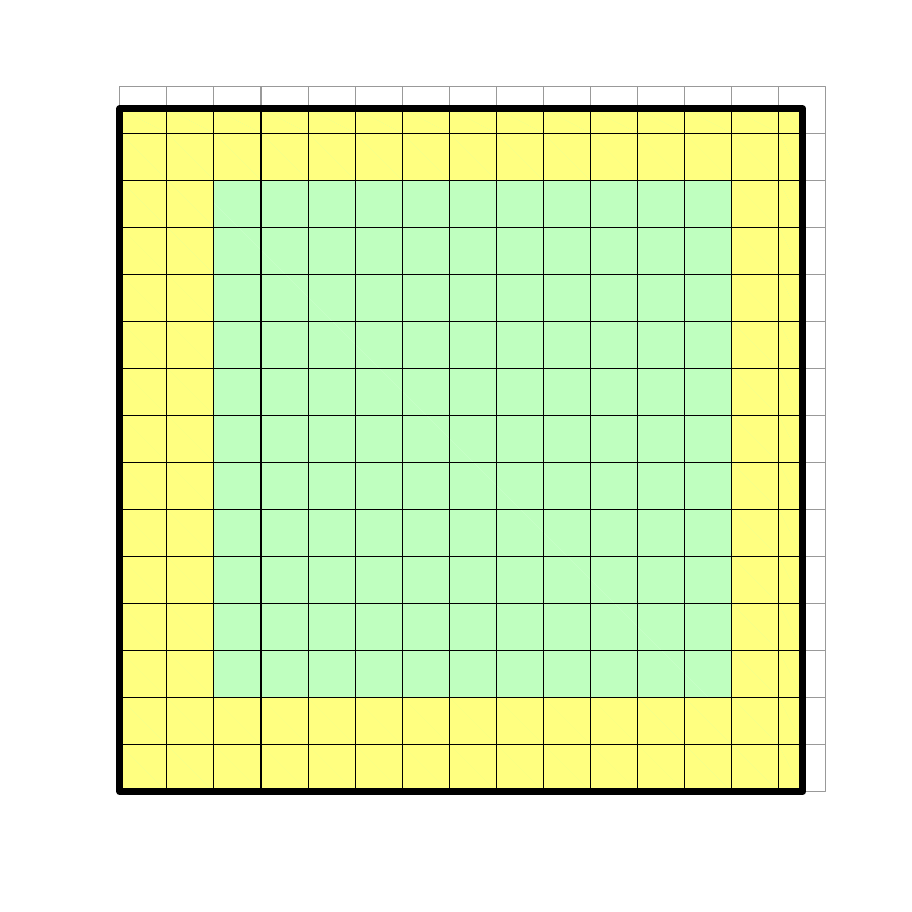}
\caption{$15\times 15$ mesh, $\delta_\mathrm{cut}=0.5$}
\end{subfigure}
\quad
\begin{subfigure}[t]{.3\linewidth}\centering
\includegraphics[width=\linewidth]{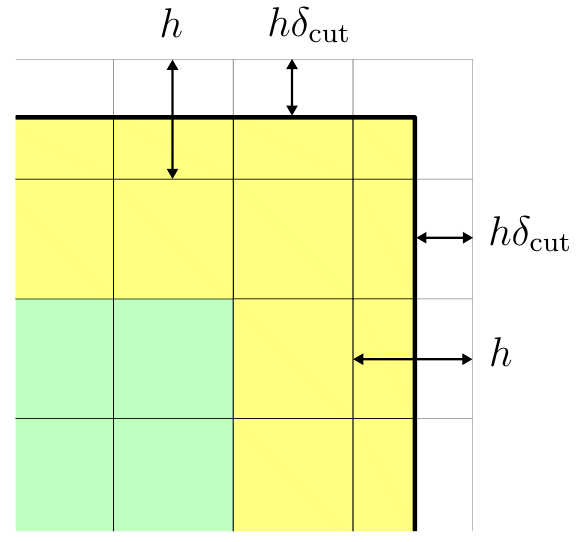}
\caption{Corner detail, $\delta_\mathrm{cut}=0.5$}
\end{subfigure}
\caption{Example mesh for the unit square model problem. Here the geometry cuts through the mesh in a controlled fashion where the topmost and rightmost elements (aside from the corner) have a proportion $\delta_\mathrm{cut}$ of their areas outside the domain. The case $\delta_\mathrm{cut}=0$ thus corresponds to a perfectly fitted mesh. The yellow part of the domain is $\mcT_{h,\delta} \cap \Omega$.}
\label{fig:unit-square-meshes}
\end{figure}

\begin{figure}
\centering
\begin{subfigure}[t]{.4\linewidth}\centering
\includegraphics[width=\linewidth]{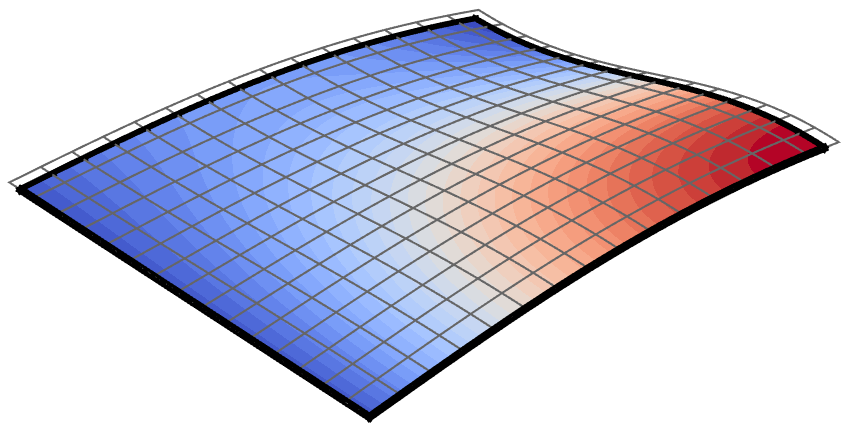}
\caption{Solution $u_h$}
\end{subfigure}
\quad
\begin{subfigure}[t]{.4\linewidth}\centering
\includegraphics[width=\linewidth]{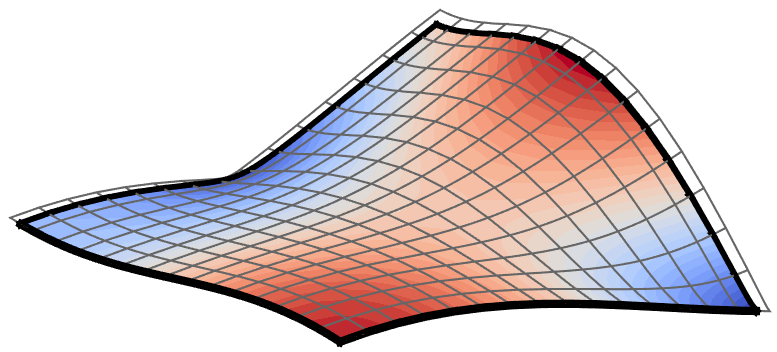}
\caption{Gradient magnitude $|\nabla u_h|$}
\end{subfigure}
\caption{Example numerical solution to the unit square model problem using the least squares stabilized Nitsche method on the mesh in Figure~\ref{fig:unit-square-meshes}.}
\label{fig:unit-square-solution}
\end{figure}

\paragraph{Unit Circle.}
As a second geometry we consider the unit circle $\Omega = \{ x \in \mathbb{R}^2 \ : \ \| x \|_{\mathbb{R}^2} \leq 1 \}$. This geometry will always produce cut situations on structured meshes. We produce a variety of cut situations by taking a background grid of size $h$ and shifting it $(th,th/3)$ where $t$ is a parameter. By letting the parameter $t$ take 100 values ranging from 0 to 1, we yield 100 different cut situations for each mesh size $h$. Some example meshes generated using this procedure are presented in Figure~\ref{fig:circle-shifted-meshes} and an example numerical solution to this model problem is shown in Figure~\ref{fig:unit-circle-solution}.

\begin{figure}
\centering
\begin{subfigure}[t]{.3\linewidth}\centering
\includegraphics[trim=30 30 20 20, clip, width=\linewidth]{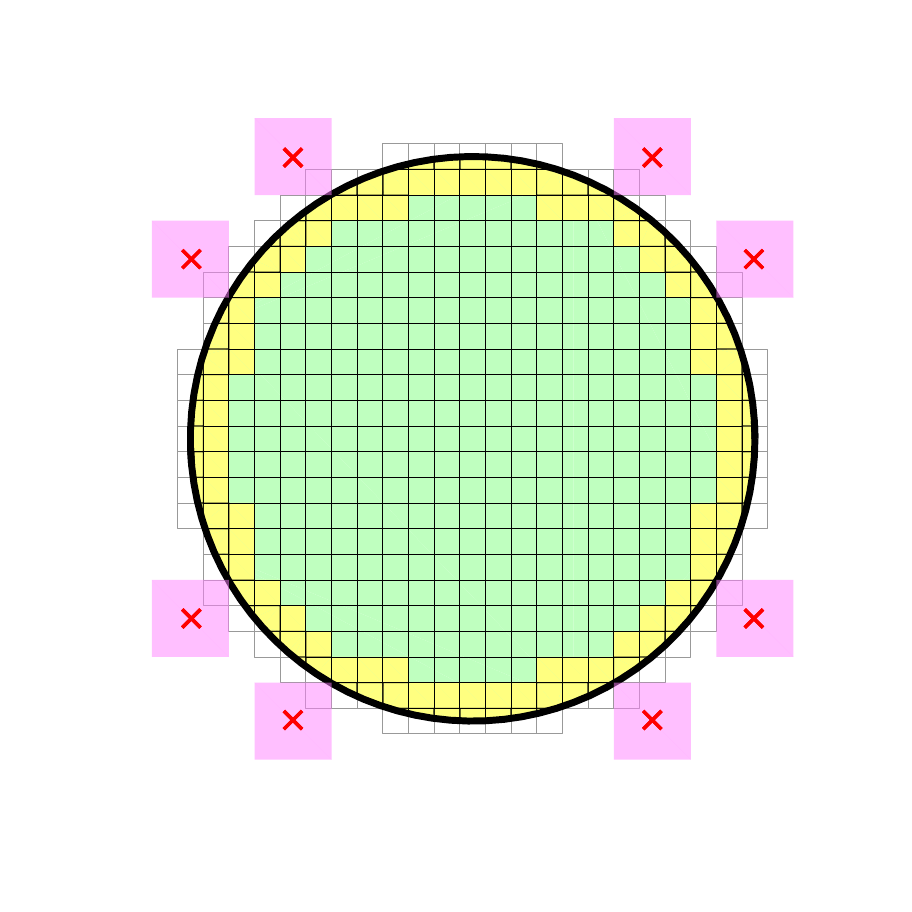}
\caption{$t=0$}
\end{subfigure}
\begin{subfigure}[t]{.3\linewidth}\centering
\includegraphics[trim=30 30 20 20, clip, width=\linewidth]{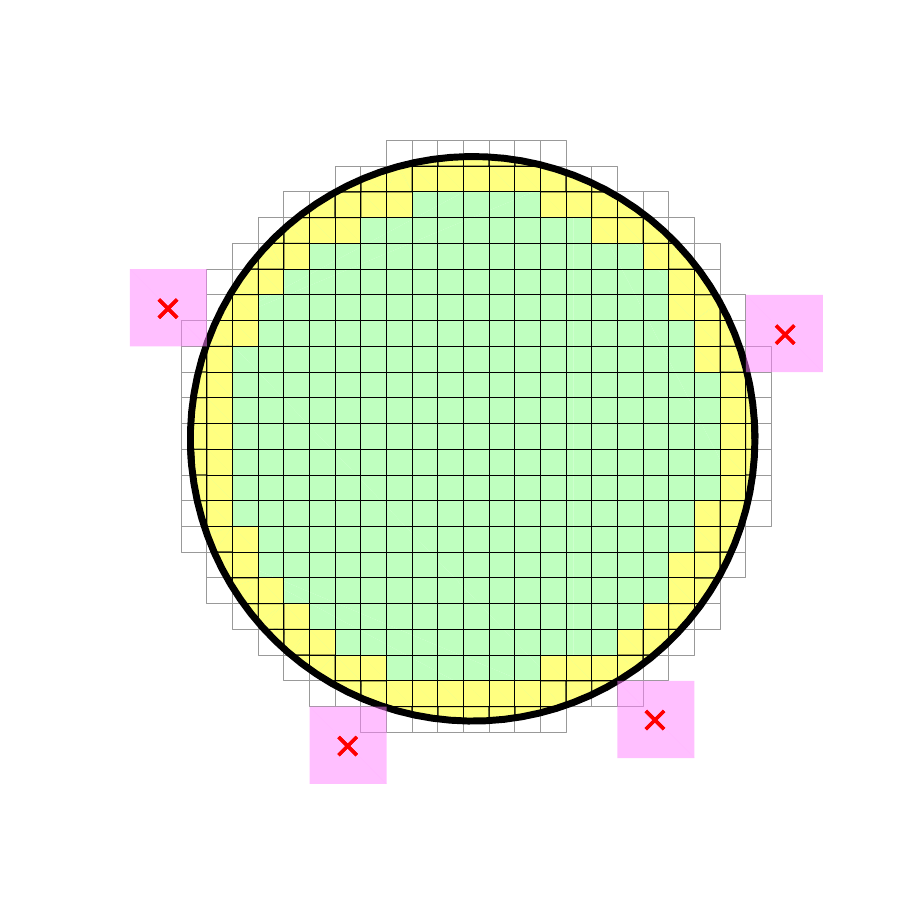}
\caption{$t=0.1$}\label{fig:that-circle-mesh}
\end{subfigure}
\begin{subfigure}[t]{.3\linewidth}\centering
\includegraphics[trim=30 30 20 20, clip, width=\linewidth]{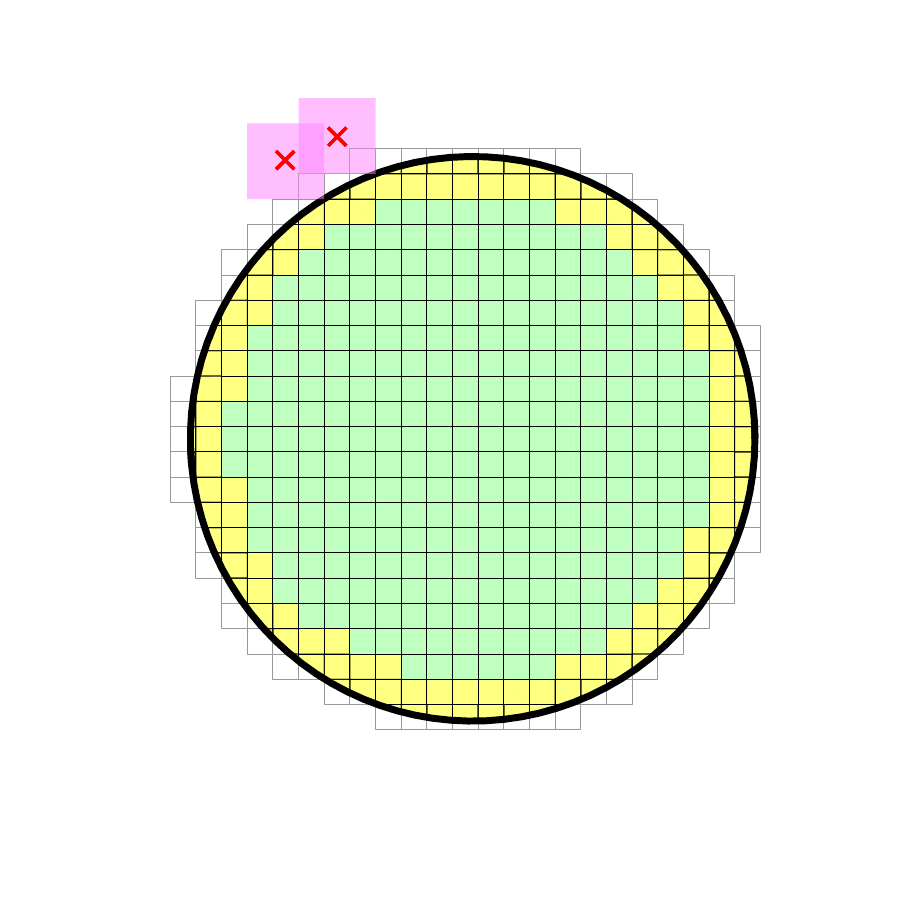}
\caption{$t=0.5$}
\end{subfigure}
\caption{Example meshes ($h=0.13$) for the unit circle. The background grid is shifted $(th,th/3)$ creating a variety of cut situations for each mesh size depending on the parameter $t$. The pink squares with a cross indicate the support of basis functions selected for basis removal using parameter values $c=0.01$ and $\tau=0.1$. The yellow part of the domain is $\mcT_{h,\delta} \cap \Omega$.}
\label{fig:circle-shifted-meshes}
\end{figure}

\begin{figure}
\centering
\begin{subfigure}[t]{.4\linewidth}\centering
\includegraphics[width=\linewidth]{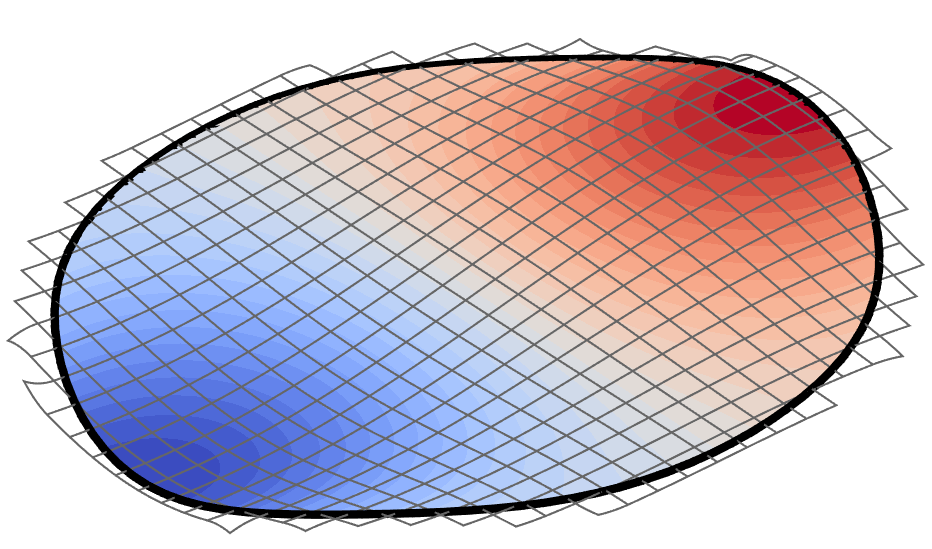}
\caption{Solution $u_h$}
\end{subfigure}
\quad
\begin{subfigure}[t]{.4\linewidth}\centering
\includegraphics[width=\linewidth]{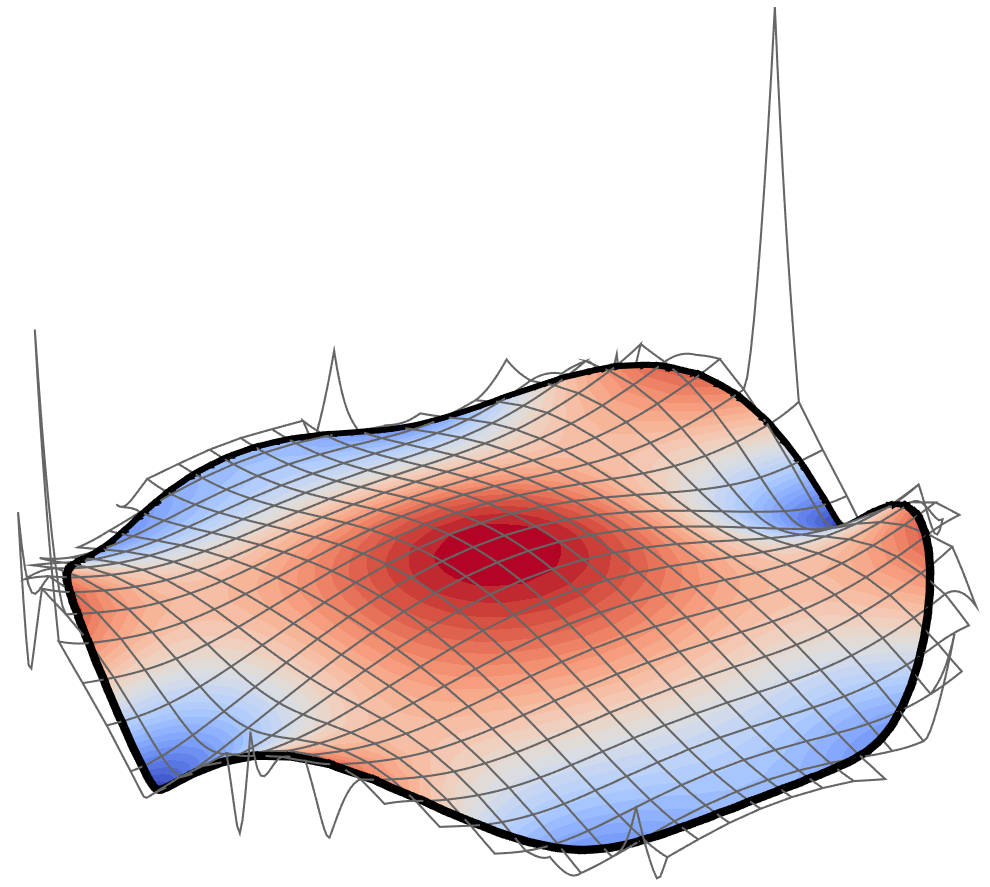}
\caption{Gradient magnitude $|\nabla u_h|$}
\end{subfigure}
\caption{Example numerical solution to the unit circle model problem using the least squares stabilized Nitsche method on the mesh in Figure~\ref{fig:that-circle-mesh}.}
\label{fig:unit-circle-solution}
\end{figure}

\subsection{Experiments}

\paragraph{Convergence Studies.}
We begin by studying convergence for the unit square problem on fitted meshes in Figure~\ref{fig:conv-fitted}. Both the least squares stabilized Nitsche method and the standard Nitsche method perform well in for the tested parameter values albeit we initially see a slightly higher $L^2(\Omega)$ error for the least squares stabilized Nitsche method for $\tau=1$. We attribute this to the $\mcT_{h,\delta}\cap\Omega$ least squares term which in addition to the imposed $h^2$ scaling also scales as the subdomain $\mcT_{h,\delta}\cap\Omega$ becomes smaller with $h$, explaining why we only note this on the coarsest meshes.

To study convergence on cut meshes we use the sequence of meshes specified for the unit circle problem. This gives access to a variety of cut situations and to illustrate the stability of the least squares stabilized Nitsche method we for each mesh size pick the largest errors among all available cut meshes, essentially producing a worst case scenario.
These convergence results are presented in Figure~\ref{fig:conv-cut-worst-case} where the least squares stabilized Nitsche method show remarkable stability. For $\tau=1$ we in the $L^2(\Omega)$ error initially note faster than expected convergence which we again attribute to the $\mcT_{h,\delta}\cap\Omega$ least squares term. For $\tau \geq 0.01$ the least squares stabilized Nitsche method and the standard Nitsche method produce visually indistinguishable convergence results. We note however that the size of the errors increase somewhat when further lowering $\tau$ and we investigate the performance of the methods for small values of $\tau$ in Figure~\ref{fig:conv-cut-worst-case-smaller-tau}.
A smaller value for $\tau$ equates to a larger effective Nitsche penalty $\beta(2+\tau^{-1})$ so we attribute this increasing error to locking due to inhomogeneous Dirichlet conditions and the fact that the boundary in the circle model problem is curved within each cut element. This latter property means that even homogeneous Dirichlet conditions cannot be exactly satisfied in our approximation space.

\begin{figure}
\centering
\begin{subfigure}[t]{.28\linewidth}
\includegraphics[trim=0 0 10 10, clip, width=\linewidth]{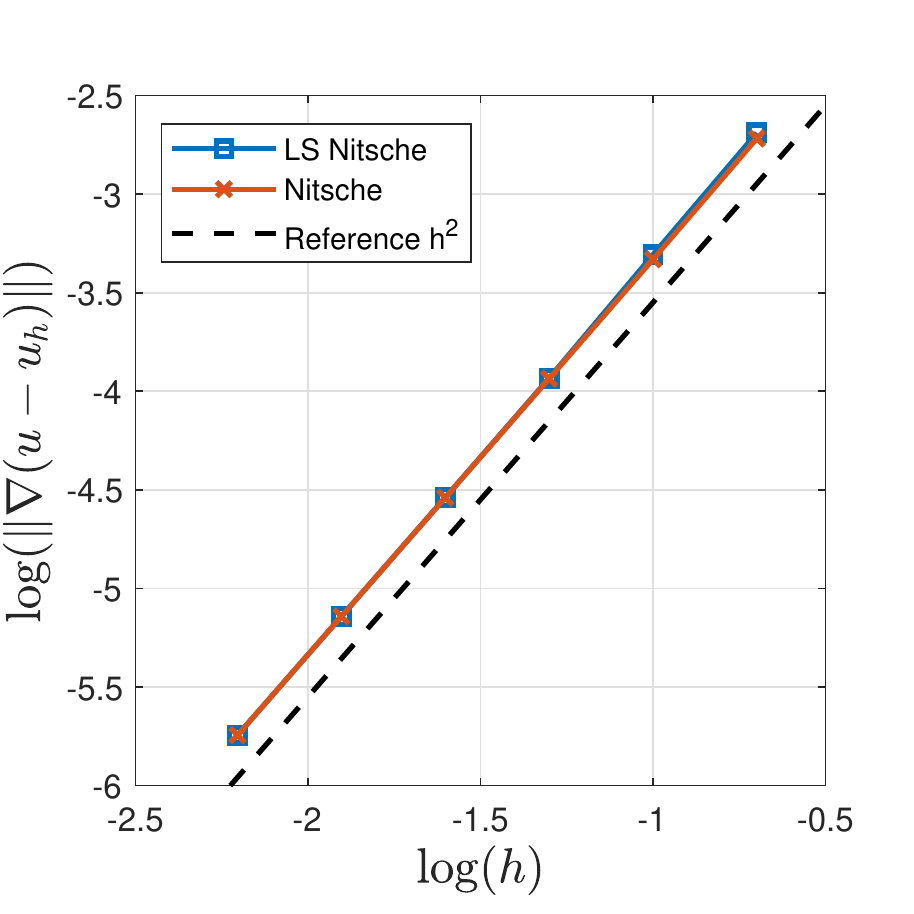}
\caption{$\tau=1$}
\end{subfigure}
\begin{subfigure}[t]{.28\linewidth}
\includegraphics[trim=0 0 10 10, clip, width=\linewidth]{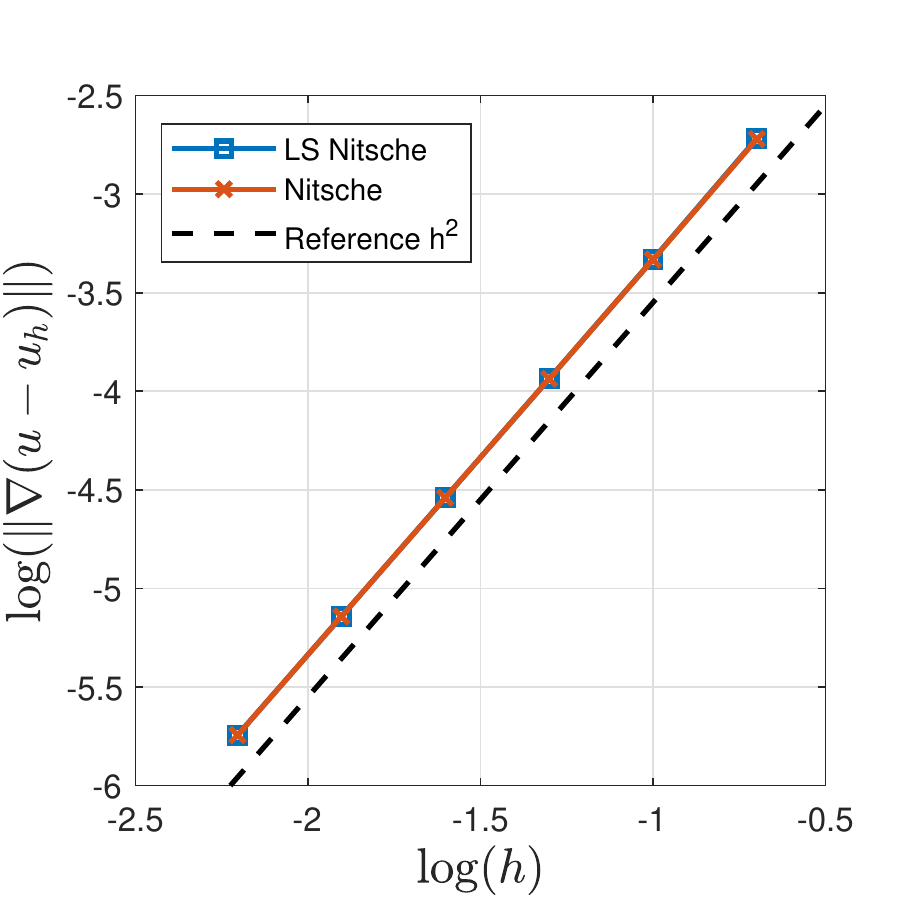}
\caption{$\tau=0.1$}
\end{subfigure}
\begin{subfigure}[t]{.28\linewidth}
\includegraphics[trim=0 0 10 10, clip, width=\linewidth]{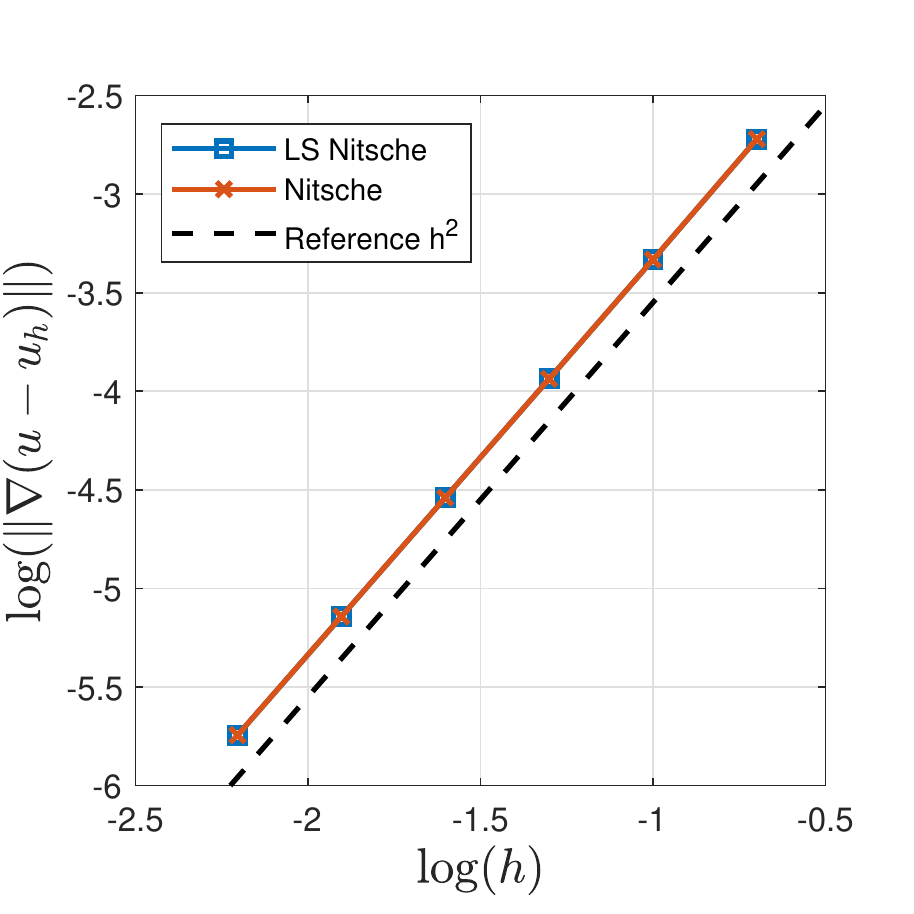}
\caption{$\tau=0.01$}
\end{subfigure}
\begin{subfigure}[t]{.28\linewidth}
\includegraphics[trim=0 0 10 10, clip, width=\linewidth]{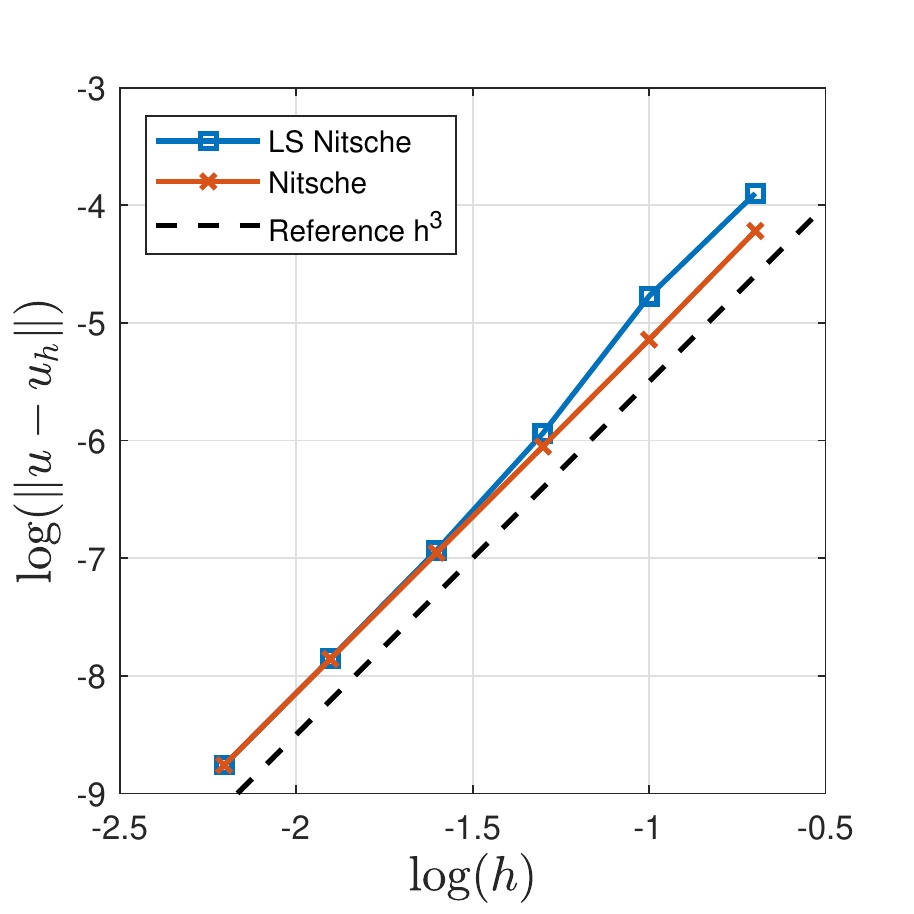}
\caption{$\tau=1$}
\end{subfigure}
\begin{subfigure}[t]{.28\linewidth}
\includegraphics[trim=0 0 10 10, clip, width=\linewidth]{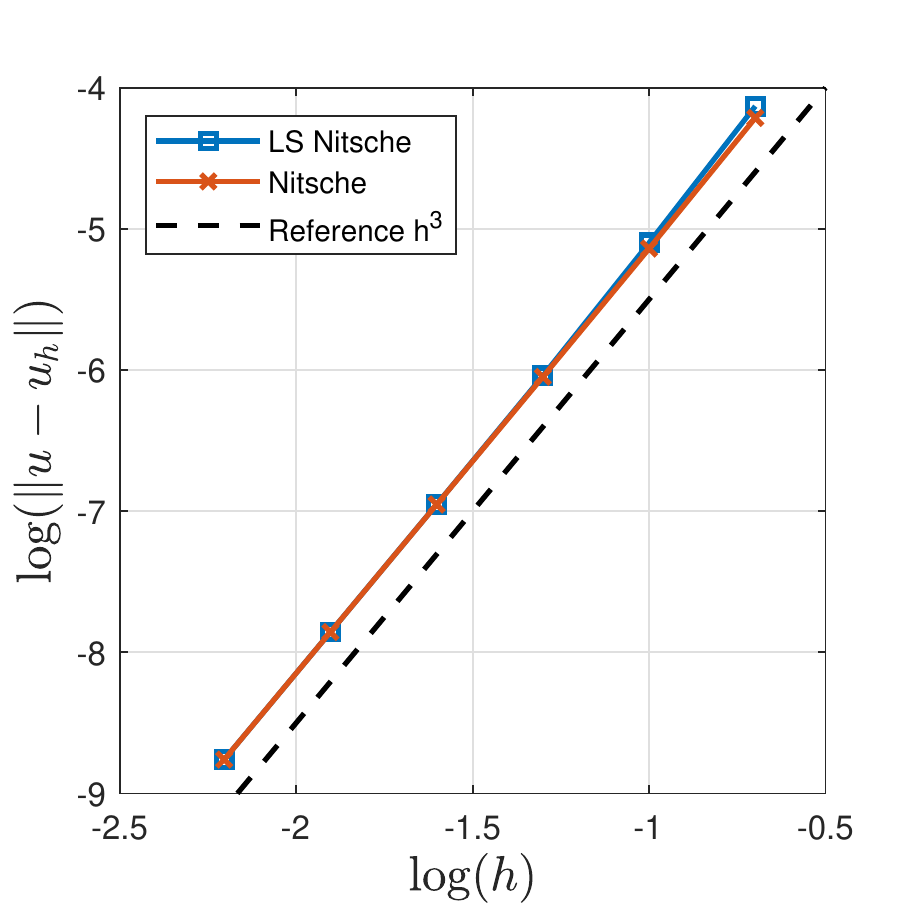}
\caption{$\tau=0.1$}
\end{subfigure}
\begin{subfigure}[t]{.28\linewidth}
\includegraphics[trim=0 0 10 10, clip, width=\linewidth]{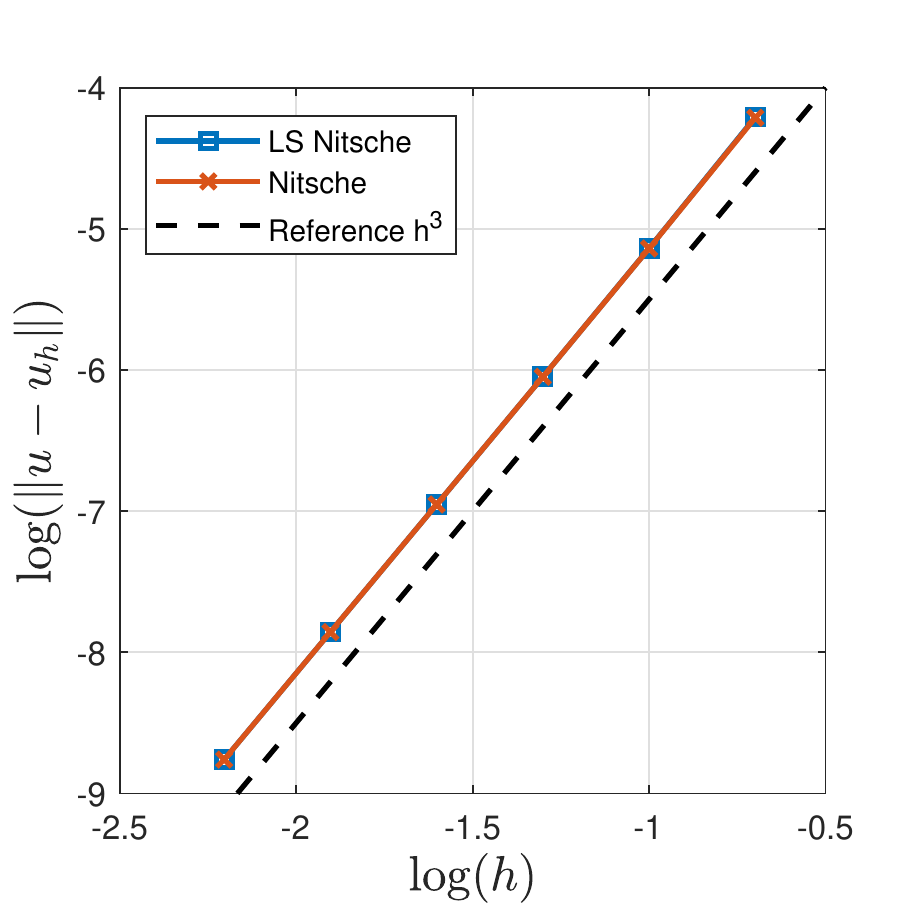}
\caption{$\tau=0.01$}
\end{subfigure}
\caption{Convergence results for the square model problem with a perfectly fitted mesh.}
\label{fig:conv-fitted}
\end{figure}

\begin{figure}
\centering
\begin{subfigure}[t]{.28\linewidth}
\includegraphics[trim=0 0 10 10, clip, width=\linewidth]{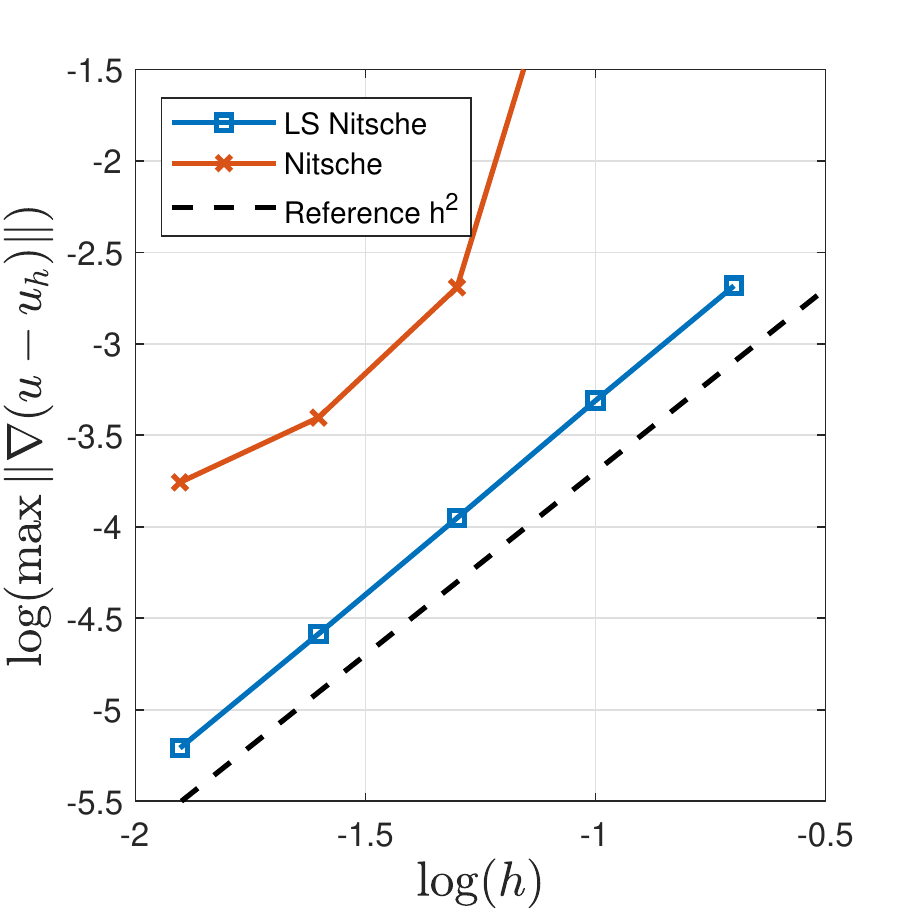}
\caption{$\tau=1$}
\end{subfigure}
\begin{subfigure}[t]{.28\linewidth}
\includegraphics[trim=0 0 10 10, clip, width=\linewidth]{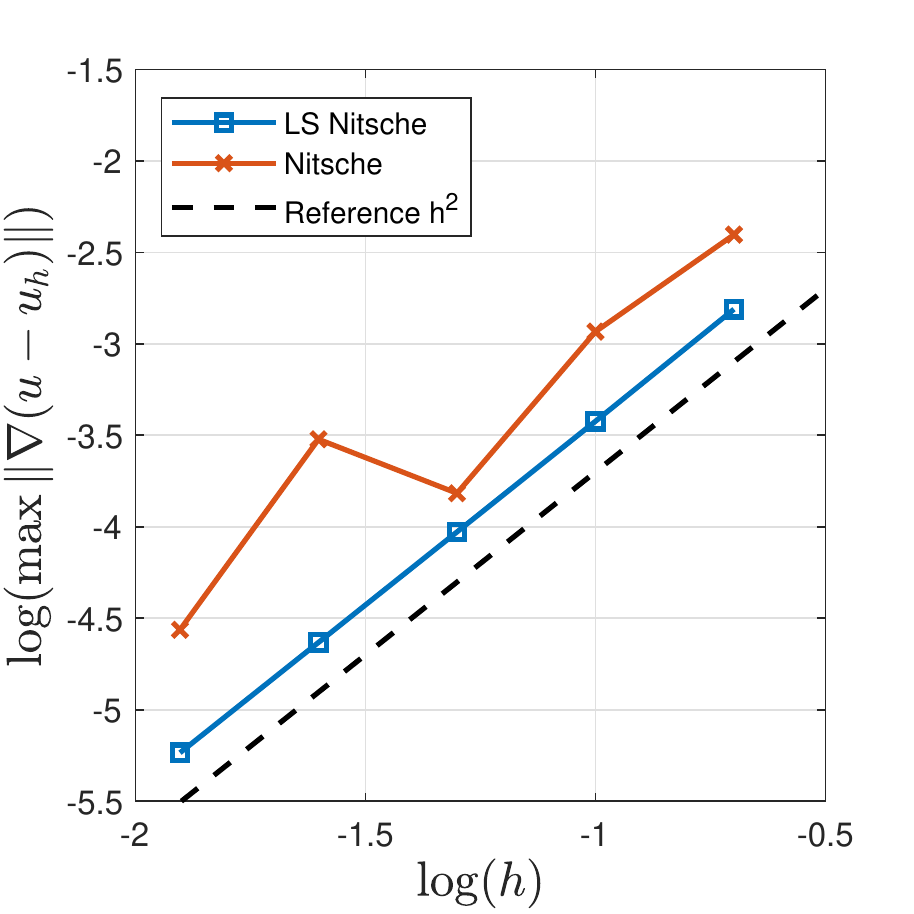}
\caption{$\tau=0.1$}
\end{subfigure}
\begin{subfigure}[t]{.28\linewidth}
\includegraphics[trim=0 0 10 10, clip, width=\linewidth]{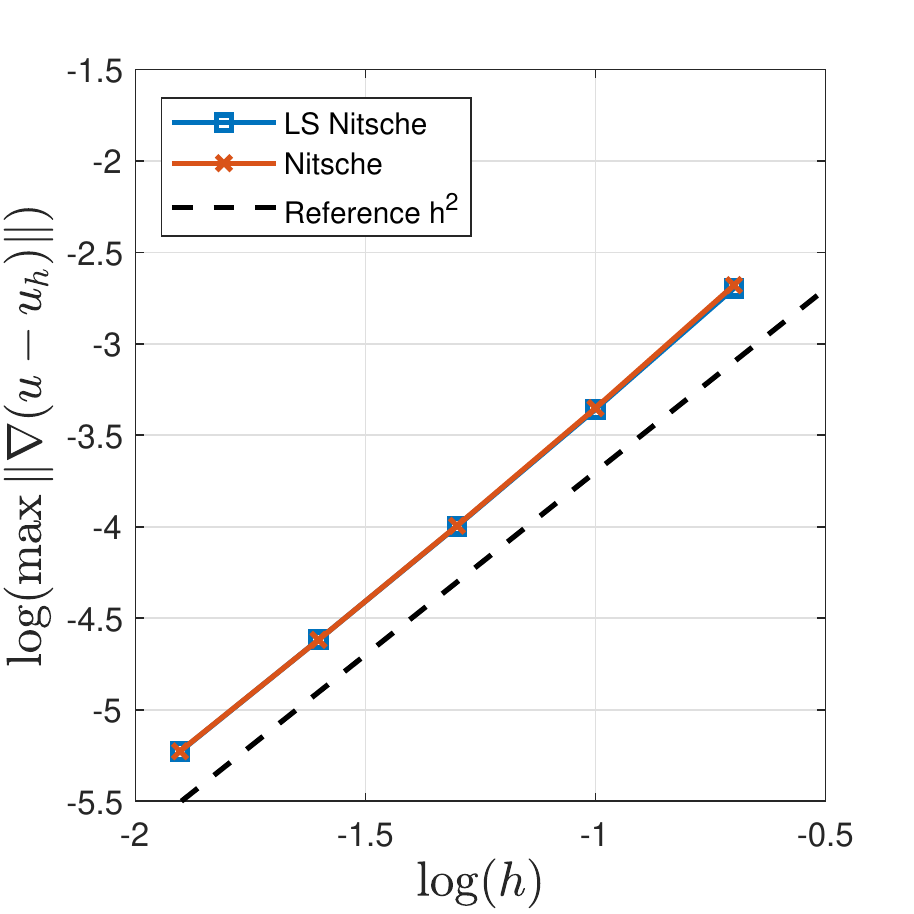}
\caption{$\tau=0.01$}
\end{subfigure}
\begin{subfigure}[t]{.28\linewidth}
\includegraphics[trim=0 0 10 10, clip, width=\linewidth]{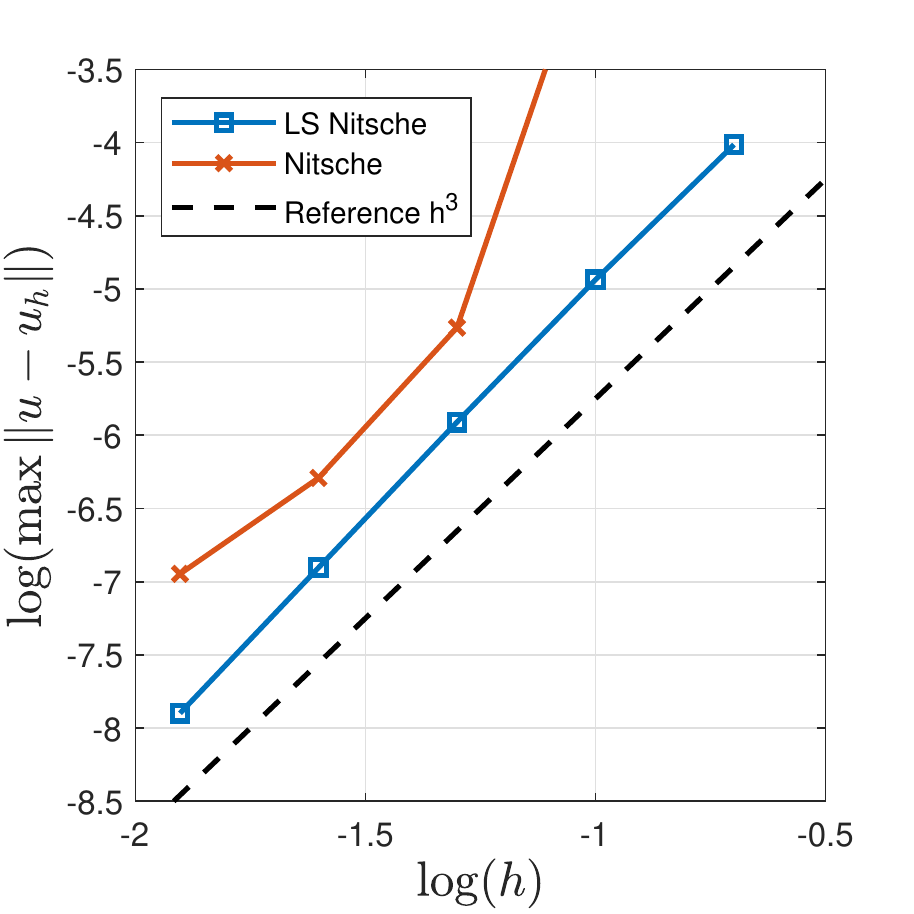}
\caption{$\tau=1$}
\end{subfigure}
\begin{subfigure}[t]{.28\linewidth}
\includegraphics[trim=0 0 10 10, clip, width=\linewidth]{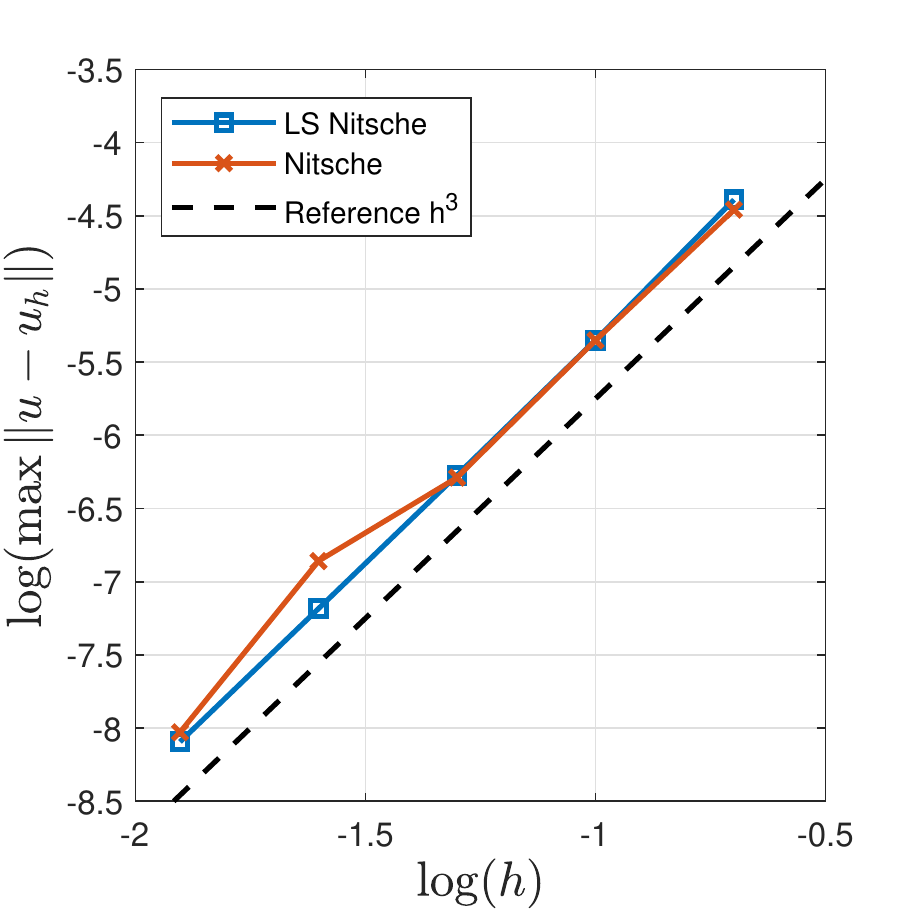}
\caption{$\tau=0.1$}
\end{subfigure}
\begin{subfigure}[t]{.28\linewidth}
\includegraphics[trim=0 0 10 10, clip, width=\linewidth]{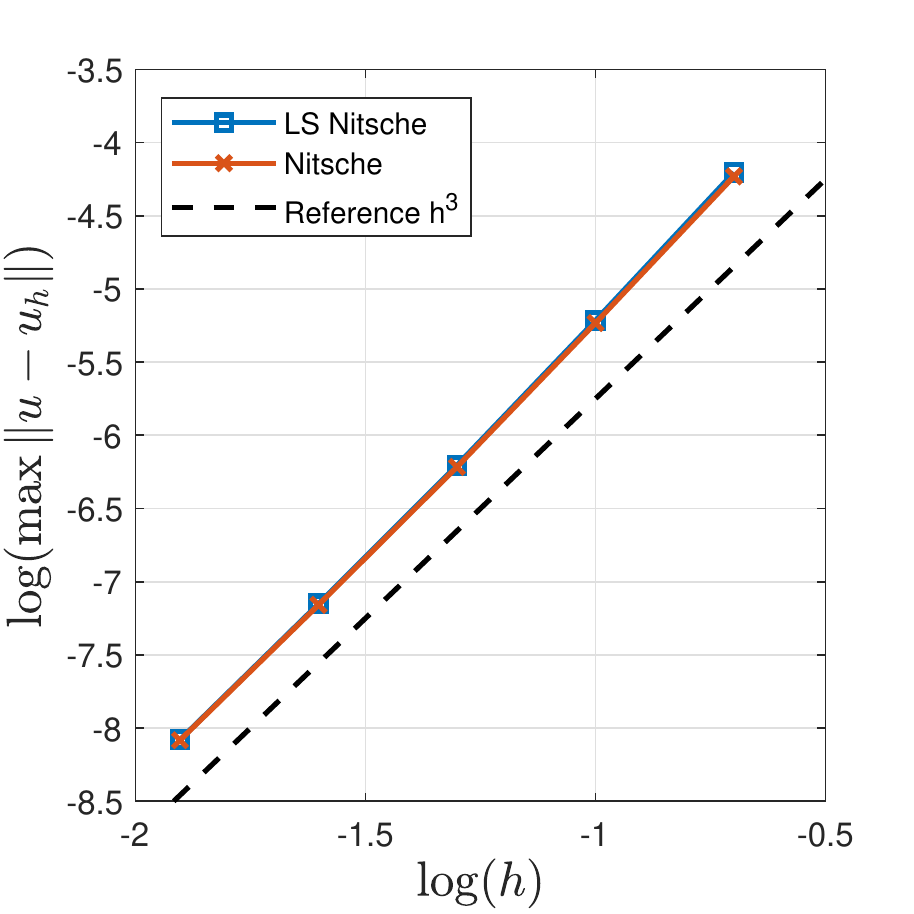}
\caption{$\tau=0.01$}
\end{subfigure}
\caption{Worst case convergence for the circle model problem on cut meshes produced by shifting the background grid to 100 different positions for each mesh size. We use $\beta=10$ and basis function removal is activated with $c=0.01$. Note that we haven't ensured coercivity for the standard Nitsche method in the worst case cut scenarios. In typical cases lack of coercivity can be remedied by increasing the penalty parameter, for example by making $\tau$ smaller, and we indeed notice improved stability of the standard Nitsche method with smaller $\tau$.}
\label{fig:conv-cut-worst-case}
\end{figure}

\begin{figure}
\centering
\begin{subfigure}[t]{.28\linewidth}
\includegraphics[trim=0 0 10 10, clip, width=\linewidth]{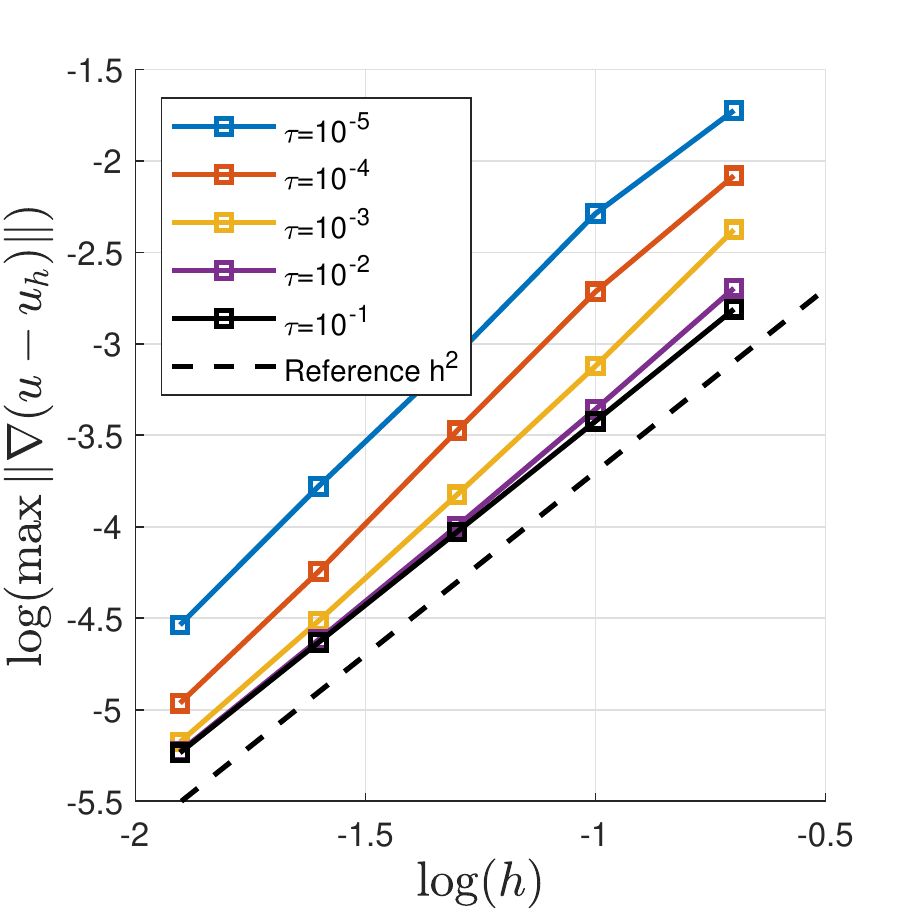}
\end{subfigure}
\begin{subfigure}[t]{.28\linewidth}
\includegraphics[trim=0 0 10 10, clip, width=\linewidth]{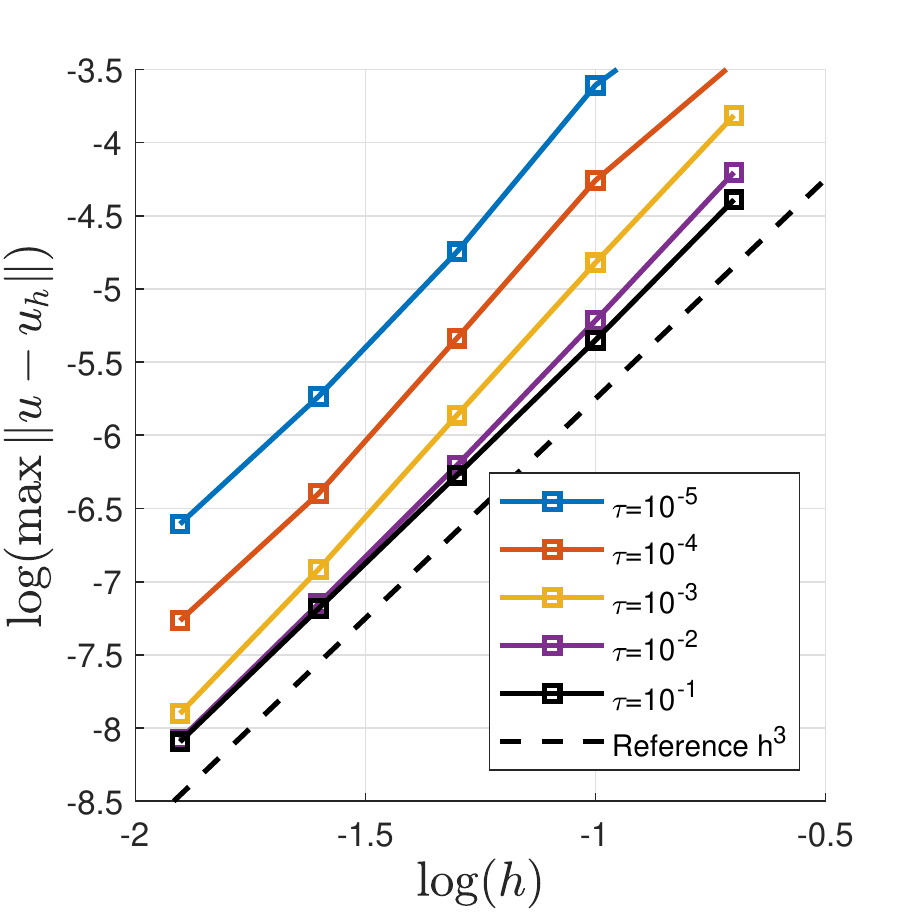}
\end{subfigure}
\begin{subfigure}[t]{.28\linewidth}
\includegraphics[trim=0 0 10 10, clip, width=\linewidth]{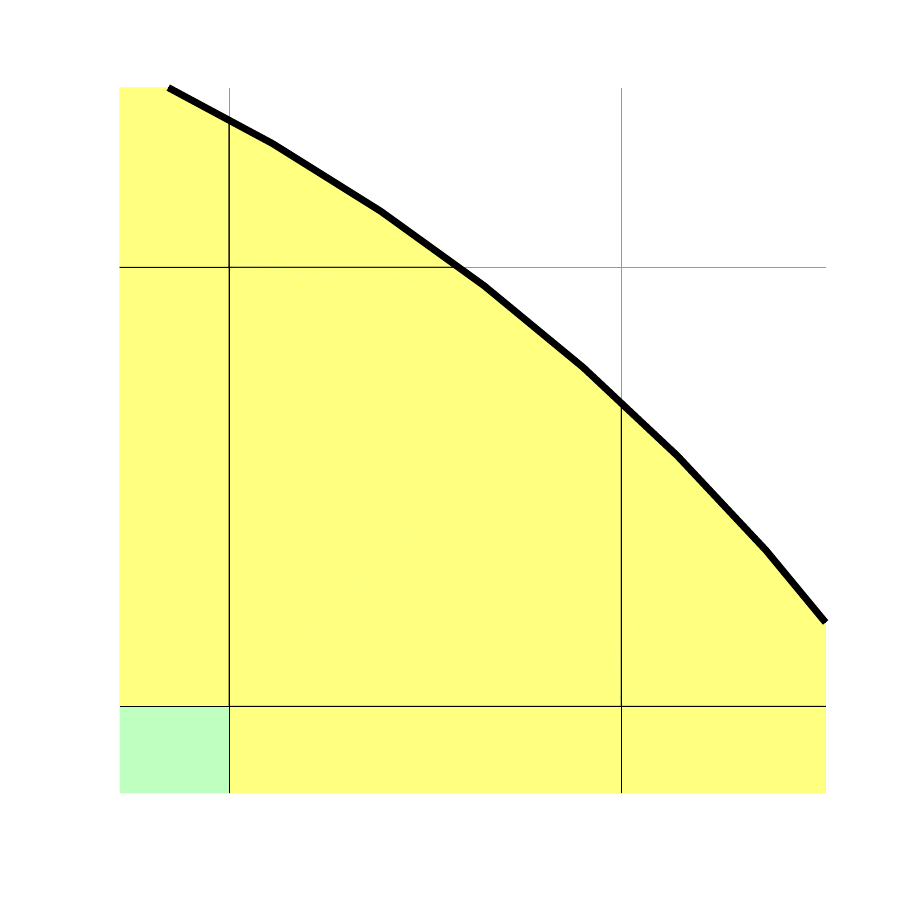}
\end{subfigure}
\caption{Worst case convergence using small values of $\tau$ in the least squares stabilized Nitsche method for the circle model problem. This corresponds to large values of the effective Nitsche parameter $\beta(2+\tau^{-1})$. We use $\beta=10$ and basis function removal is activated with $c=0.01$. The standard Nitsche method yields close to identical results for these values of $\tau$ with the exception of $\tau=0.1$ which is studied in Figure~\ref{fig:conv-cut-worst-case}. We attribute the increasing errors when lowering $\tau$ to locking which is pronounced by the boundary being curved within cut elements as shown in the mesh detail on the right.}
\label{fig:conv-cut-worst-case-smaller-tau}
\end{figure}

\paragraph{Condition Number Study.}
To illustrate the effect of basis function removal on the stiffness matrix condition number we use the unit circle model problem on a background grid of size $h=0.13$ and shift this background grid in 100 steps as described above to create a variety of cut situations. The results from this study are presented in Figure~\ref{fig:condest}. We note that the effect of basis function removal on the least squares stabilized Nitsche method seems stable for the tested values of $\tau$ while we note some minor instabilities when applying basis function removal on the standard Nitsche method for $\tau \in \{1,0.1\}$.

\begin{figure}
\centering
\begin{subfigure}[t]{.28\linewidth}
\includegraphics[trim=0 0 10 10, clip, width=\linewidth]{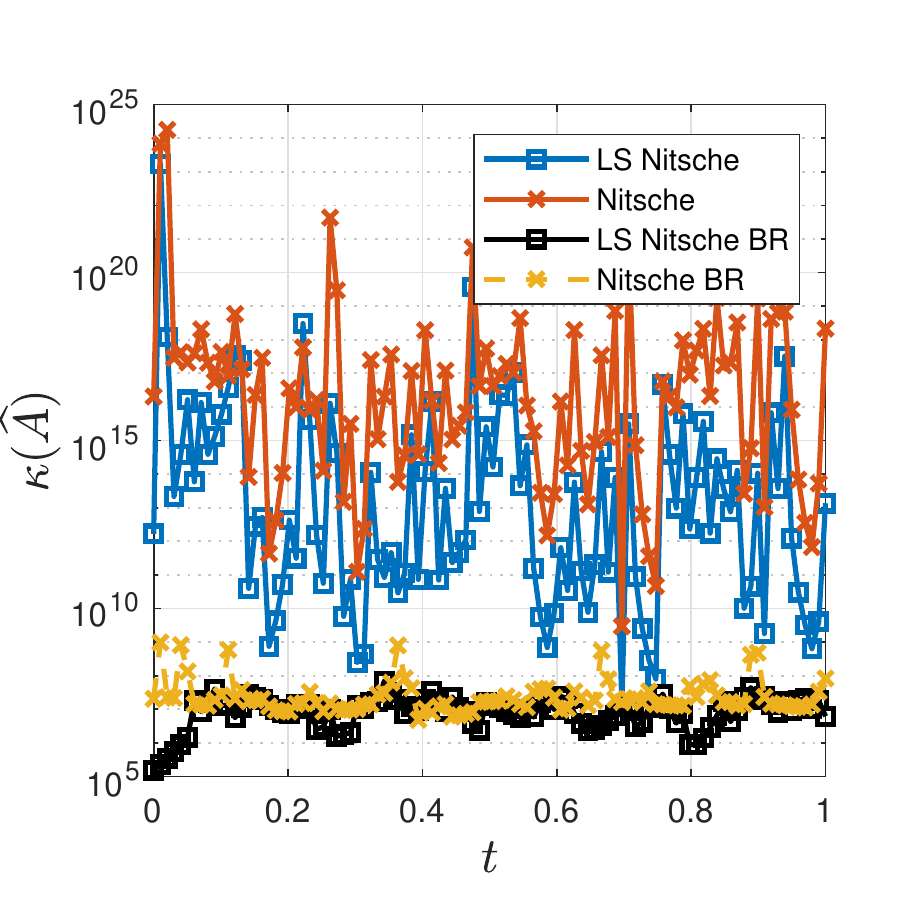}
\caption{$\tau=1$}
\end{subfigure}
\begin{subfigure}[t]{.28\linewidth}
\includegraphics[trim=0 0 10 10, clip, width=\linewidth]{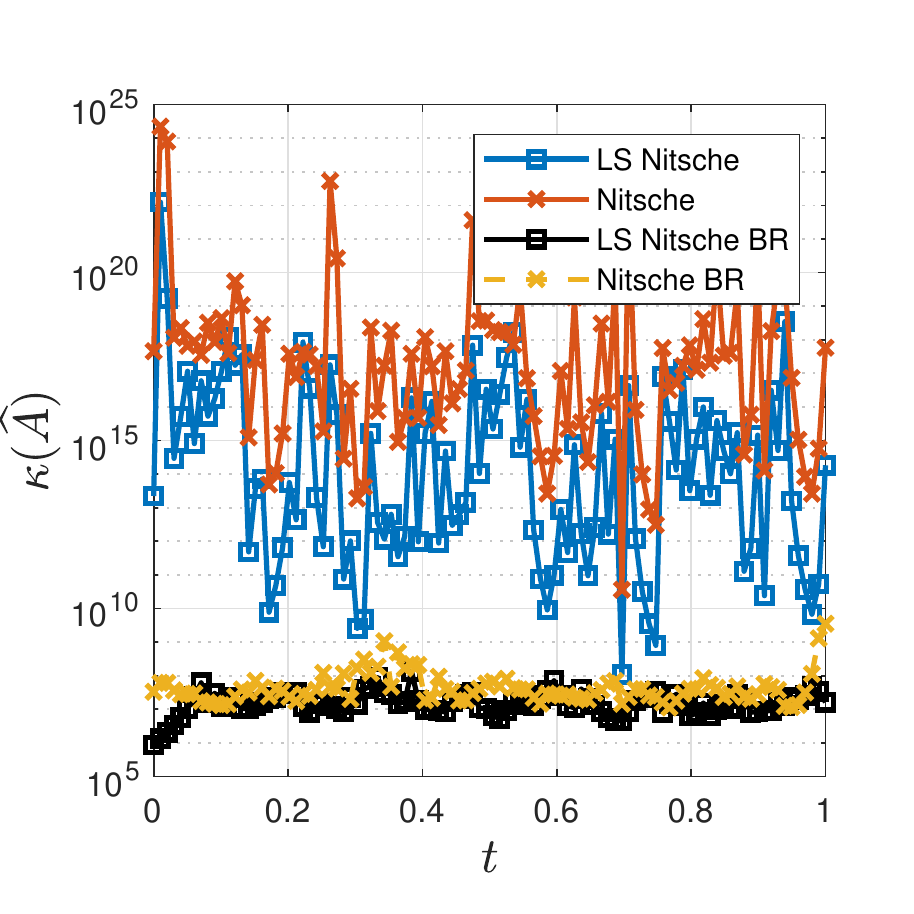}
\caption{$\tau=0.1$}
\end{subfigure}
\begin{subfigure}[t]{.28\linewidth}
\includegraphics[trim=0 0 10 10, clip, width=\linewidth]{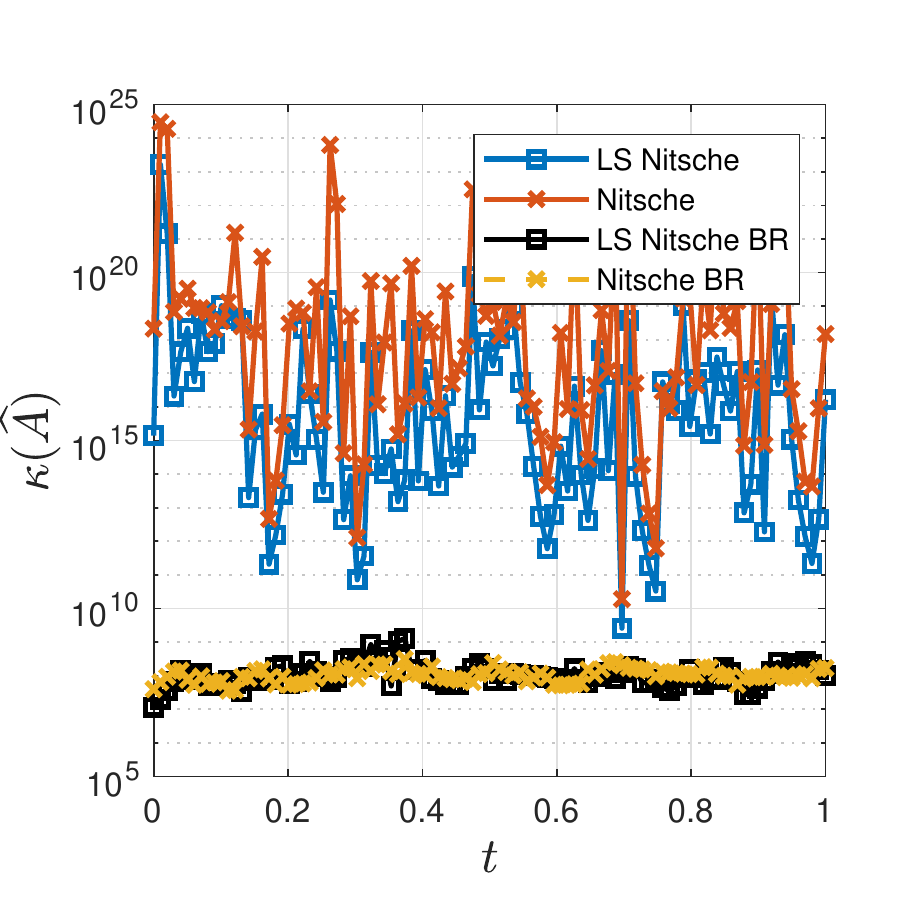}
\caption{$\tau=0.01$}
\end{subfigure}
\caption{Stiffness matrix condition numbers in the unit circle model problem where the background grid is shifted based on the parameter $t$ creating a variety of cut situations. Condition numbers after basis function removal with $c=0.1$ are marked ``BR".}
\label{fig:condest}
\end{figure}

\paragraph{Coercivity on $\boldsymbol H^{\boldsymbol 2}\boldsymbol(\boldsymbol\Omega\boldsymbol)$.}
The least squares stabilized Nitsche method is coercive on the full space $V=H^2(\Omega)$ rather than only on the finite element space $V_h$, which is conventionally the case for Nitsche methods. To illustrate this we keep the method fixed and study behavior of the smallest eigenvalues when refining the mesh. If the smallest eigenvalue is negative the method cannot be coercive.
We fix the method based on a $10\times 10$ mesh, meaning the parameter $h$ and subdomain $\mcT_{h,\delta}$ are fixed in the method \eqref{eq:method} and are thereby independent of the actual computational mesh, see Figure~\ref{fig:meshes-fixed-method}.
As expected the results in Figure~\ref{fig:smev-fixed-method} show that the least squares stabilized Nitsche method maintains a positive smallest eigenvalue in every investigated case. The size of the smallest eigenvalue however approaches zero. The standard Nitsche method is only coercive on $V_h$ and thus eventually attains negative smallest eigenvalues if the method is held fixed.

\begin{figure}
\centering
\begin{subfigure}[t]{.23\linewidth}\centering
\includegraphics[trim=15 15 10 10, clip, width=0.91\linewidth]{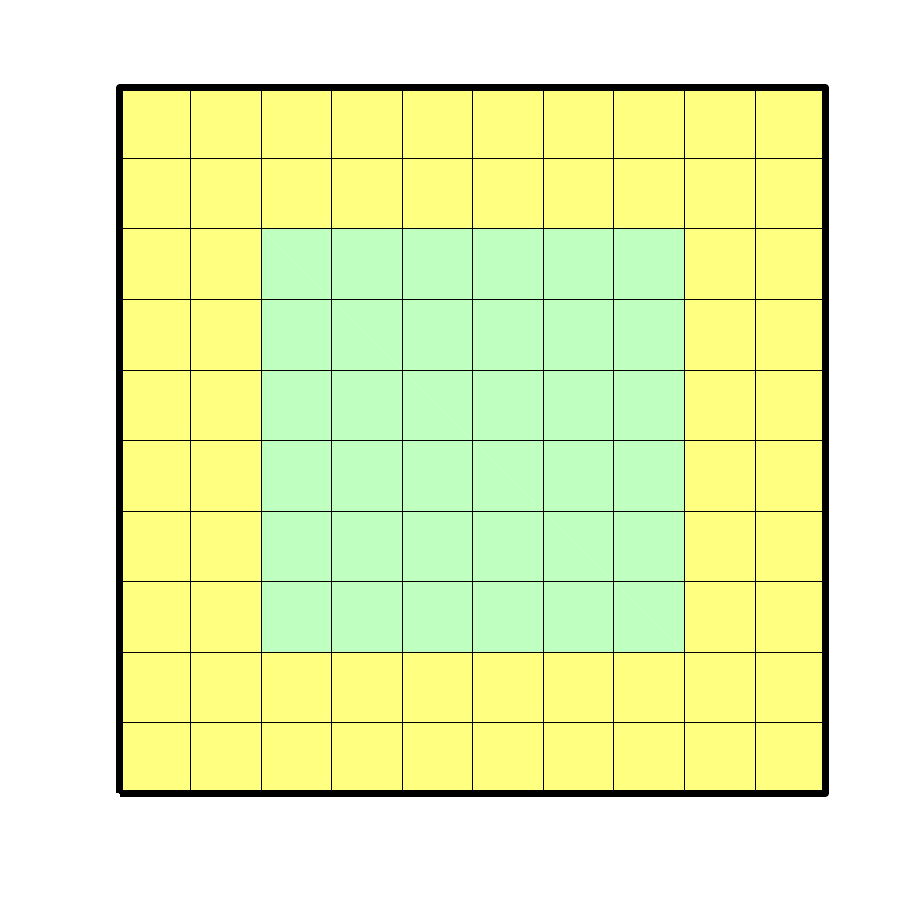}
\caption{$\delta_\mathrm{cut}=0$}
\end{subfigure}
\begin{subfigure}[t]{.23\linewidth}\centering
\includegraphics[trim=15 15 10 10, clip, width=0.91\linewidth]{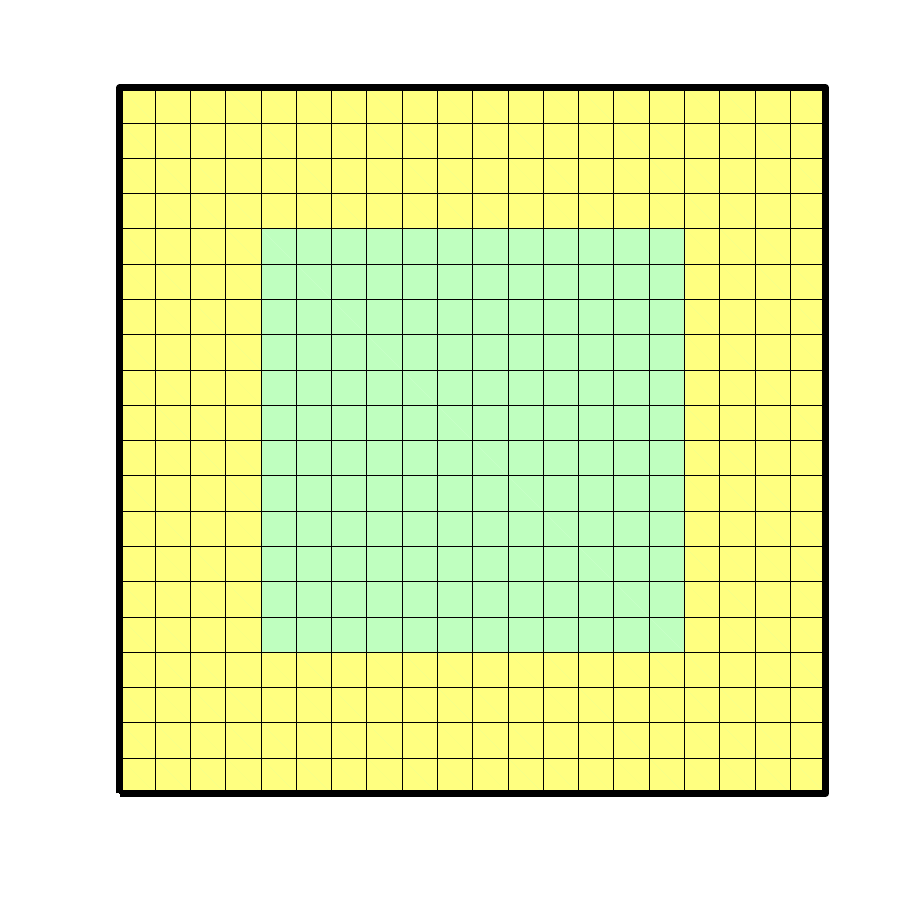}
\caption{$\delta_\mathrm{cut}=0$}
\end{subfigure}\ \
\begin{subfigure}[t]{.23\linewidth}\centering
\includegraphics[trim=15 15 10 10, clip, width=\linewidth]{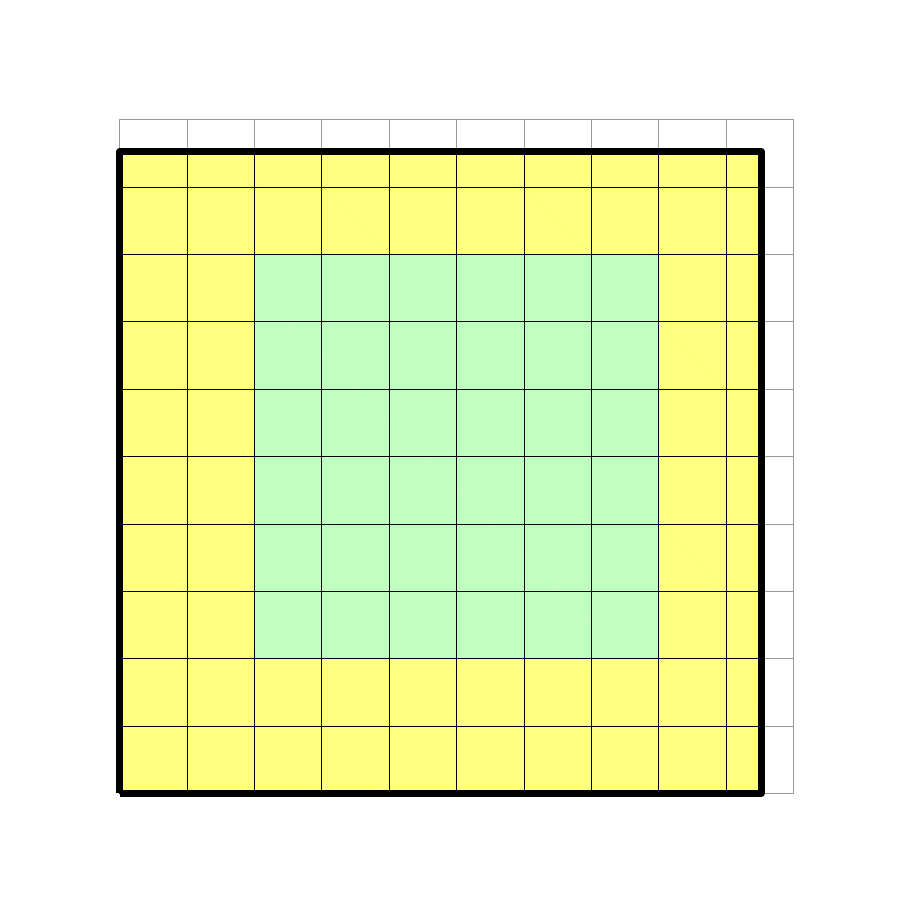}
\caption{$\delta_\mathrm{cut}=0.5$}
\end{subfigure}
\begin{subfigure}[t]{.23\linewidth}\centering
\includegraphics[trim=15 15 10 10, clip, width=\linewidth]{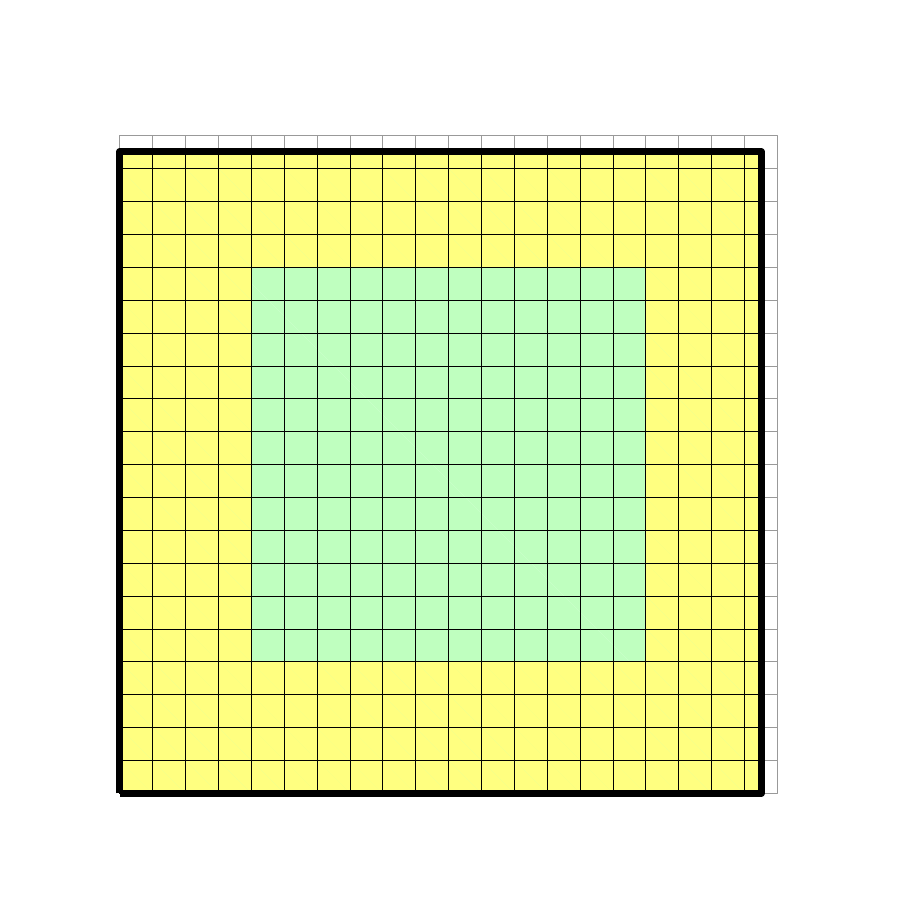}
\caption{$\delta_\mathrm{cut}=0.5$}
\end{subfigure}
\caption{$10\times 10$ and $20\times 20$ meshes with the method fixed at $10\times 10$. Yellow indicates the fixed subdomain $\mcT_{h,\delta} \cap \Omega$ in the method. (a)--(b) One refinement of a fitted mesh. (c)--(d) One refinement of a cut mesh, preserving $\delta_\mathrm{cut}=0.5$. In this construction of cut meshes $\mcT_{h,\delta} \cap \Omega$ is not precisely fixed albeit converges to the corresponding domain in the fitted grid case.
}
\label{fig:meshes-fixed-method}
\end{figure}

\begin{figure}
\centering
\begin{subfigure}[t]{.28\linewidth}
\includegraphics[trim=0 0 5 10, clip, width=\linewidth]{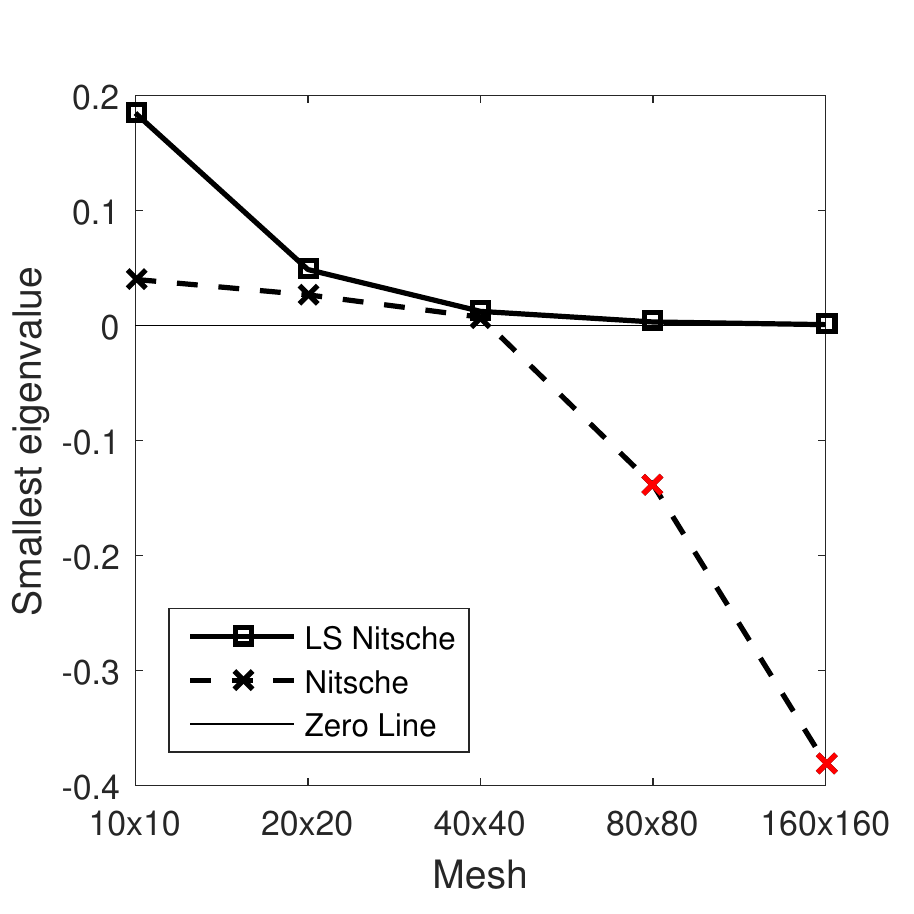}
\caption{$\delta_\mathrm{cut}=0$, $\tau=1$}
\end{subfigure}
\begin{subfigure}[t]{.28\linewidth}
\includegraphics[trim=0 0 5 10, clip, width=\linewidth]{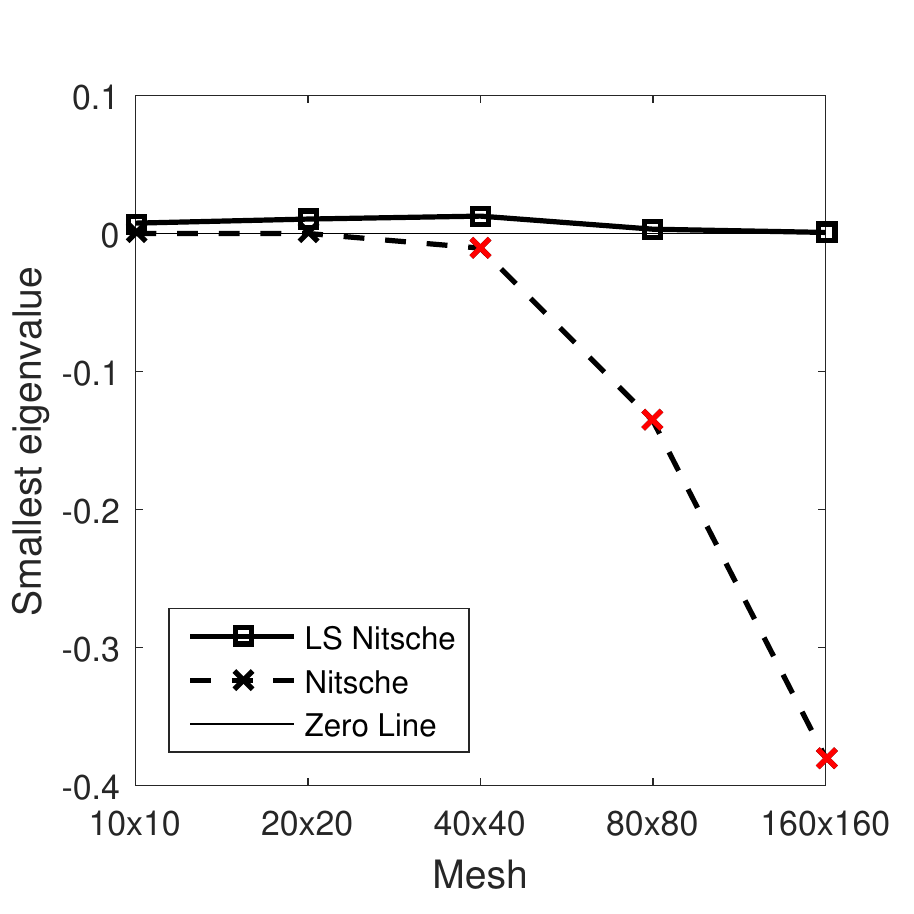}
\caption{$\delta_\mathrm{cut}=0.5$, $\tau=1$}
\end{subfigure}
\begin{subfigure}[t]{.28\linewidth}
\includegraphics[trim=0 0 5 10, clip, width=\linewidth]{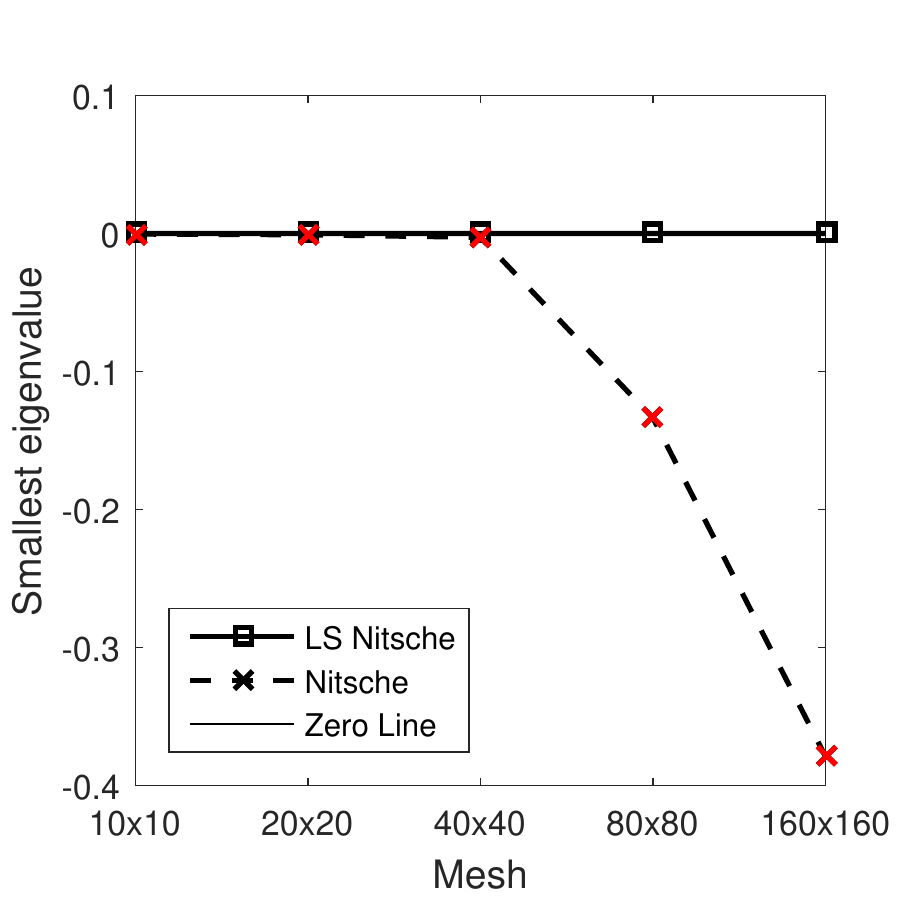}
\caption{$\delta_\mathrm{cut}=0.9$, $\tau=1$}
\end{subfigure}

\begin{subfigure}[t]{.28\linewidth}
\includegraphics[trim=0 0 5 10, clip, width=\linewidth]{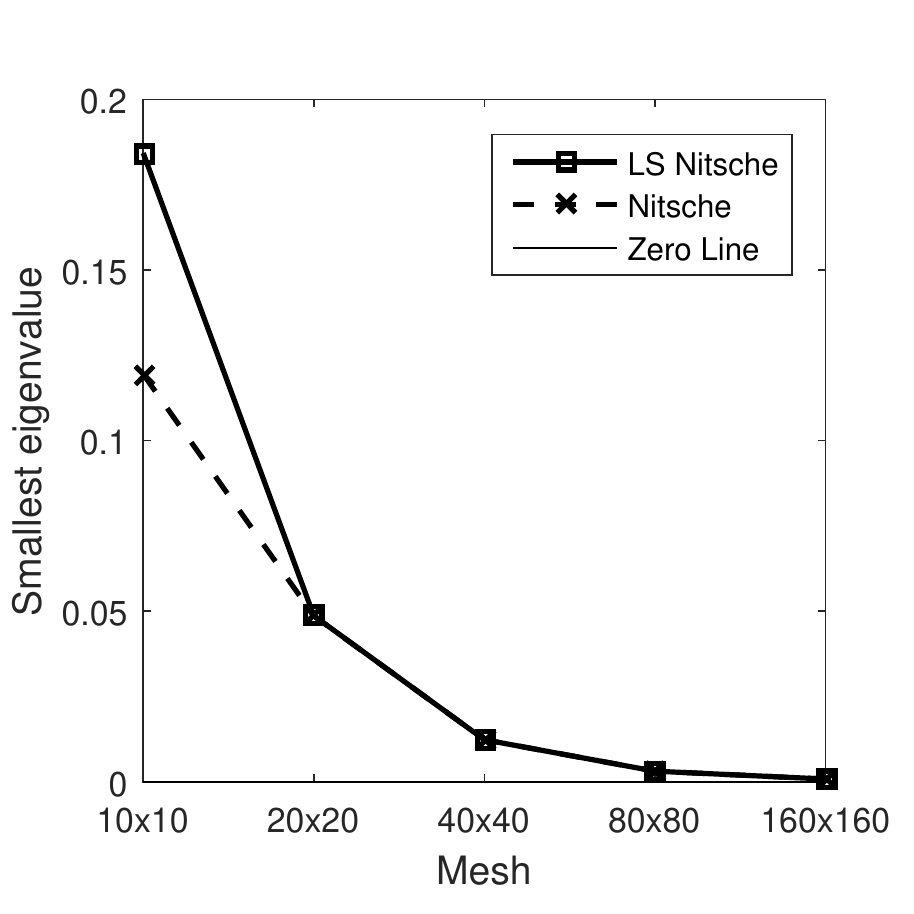}
\caption{$\delta_\mathrm{cut}=0$, $\tau=0.1$}
\end{subfigure}
\begin{subfigure}[t]{.28\linewidth}
\includegraphics[trim=0 0 5 10, clip, width=\linewidth]{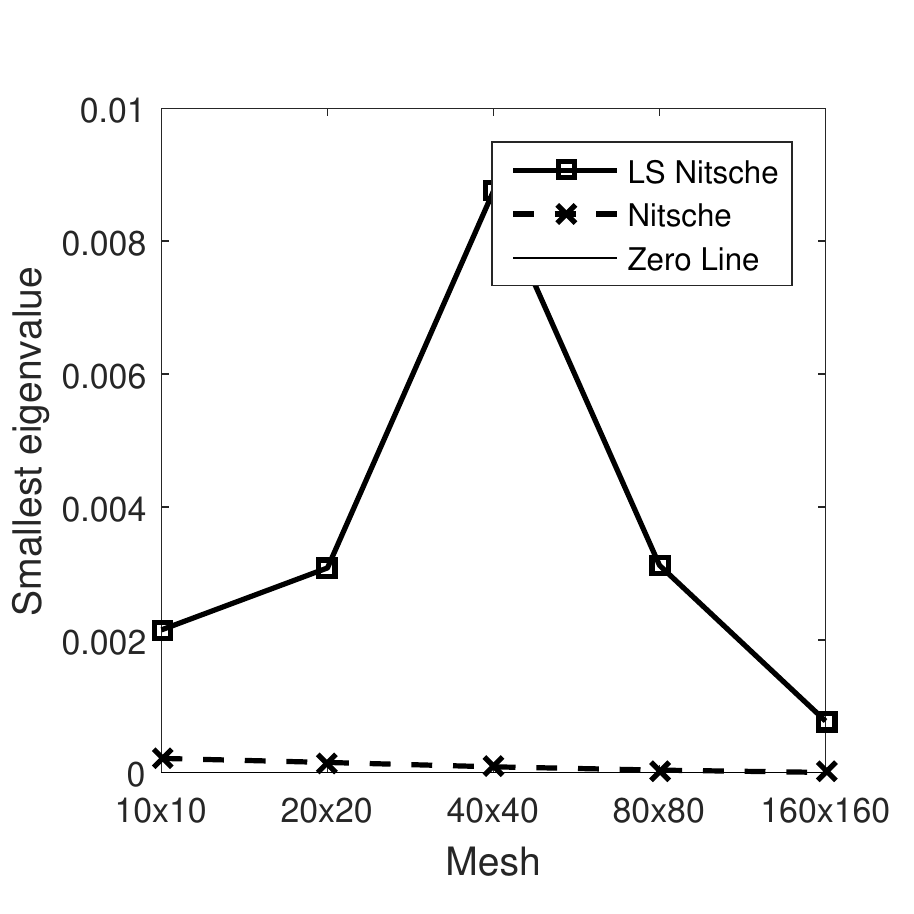}
\caption{$\delta_\mathrm{cut}=0.5$, $\tau=0.1$}
\end{subfigure}
\begin{subfigure}[t]{.28\linewidth}
\includegraphics[trim=0 0 5 10, clip, width=\linewidth]{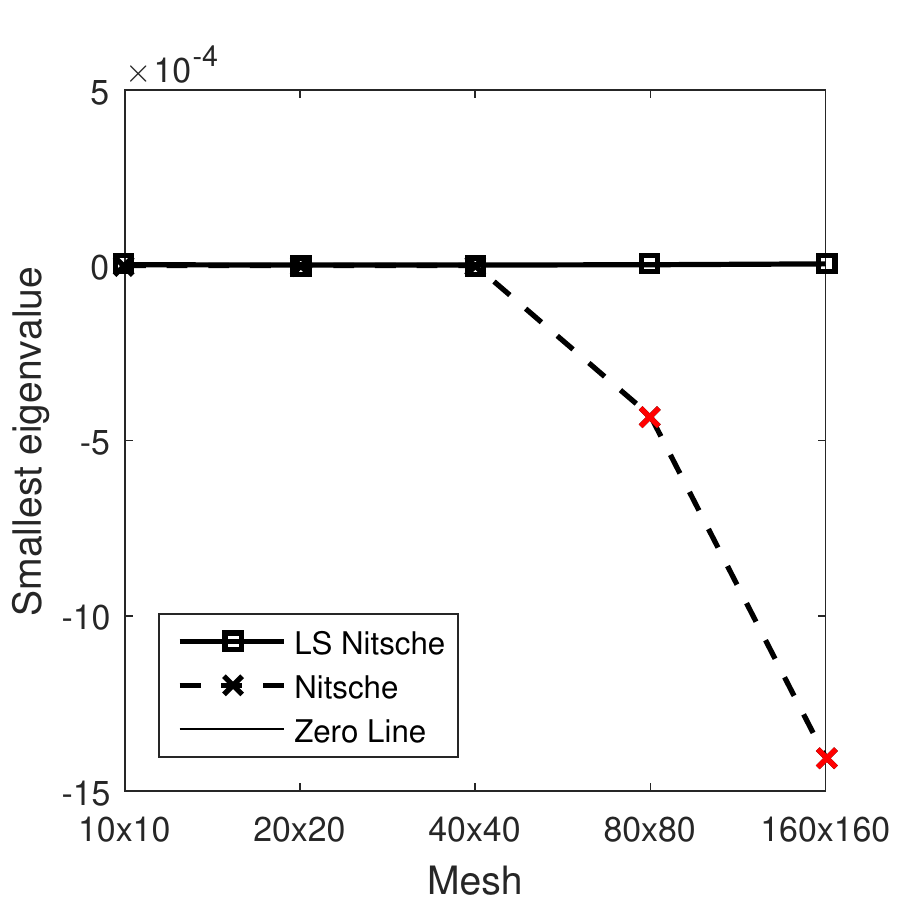}
\caption{$\delta_\mathrm{cut}=0.9$, $\tau=0.1$}
\end{subfigure}

\caption{Smallest stiffness matrix eigenvalues versus mesh refinements when the method is fixed at $10\times 10$. That means that $\mcT_{h,\delta}\cap\Omega$ and $h$ are fixed in the method and thus independent of the choice of mesh. Negative eigenvalues are displayed in red. The least squares stabilized Nitsche retains a positive smallest eigenvalue in all experiments.}
\label{fig:smev-fixed-method}
\end{figure}

\section{Conclusions}

We have developed a new symmetric Nitsche formulation for cut $C^1$ elements that is coercive on $H^2(\Omega)$ and that does not rely on ghost penalties or choosing a very large penalty parameter in the Nitsche penalty term. Instead, the method is based on adding certain consistent least squares terms on, and in the vicinity of, the boundary. This new least square stabilized symmetric Nitsche method has the following notable features:
\begin{itemize}
\item The least squares stabilization terms are consistent and are only added elementwise.

\item The method is coercive with respect to the energy norm on the full space $V=H^2(\Omega)$ rather than on the discrete space $V_h$ as is the case in the analysis of the standard symmetric Nitsche method which rely on inverse inequalities on $V_h$.

\item Since the intersection between elements and the domain $\Omega$ may be arbitrary small the $H^2(\Omega)$ coercivity only guarantee that the stiffness matrix is positive semidefinite. To ensure a positive definite stiffness matrix we employ the Basis Function Removal technique recently introduced in \cite{ElfLarLar18}.

\item In the numerical results we achieve very stable convergence results with respect to the cut situation even when using a penalty parameter which is only of moderate size. This is desirable in many cases where choosing a too large penalty parameter may introduce locking, for example when the boundary is non-trivial within cut elements, when using inhomogeneous boundary data, or in interface problems on non-matching grids.
The $\tau$ parameter in the formulation allows for convenient adjustment of the penalty parameter size in the Nitsche penalty term while maintaining coercivity.

\end{itemize}

\bigskip
\paragraph{Acknowledgements.}
This research was supported in part by the Swedish Foundation for Strategic Research Grant No.\ AM13-0029, the Swedish Research Council Grants Nos.\ 2013-4708, 2017-03911 and the Swedish Research Programme Essence.


\bibliographystyle{abbrv}
\footnotesize{
\bibliography{refs}
}

\bigskip
\bigskip
\noindent
\footnotesize {\bf Authors' addresses:}

\smallskip

\noindent
Daniel Elfverson,  \quad \hfill \addressumushort\\
{\tt daniel.elfverson@umu.se}

\smallskip
\noindent
Mats G. Larson,  \quad \hfill \addressumushort\\
{\tt mats.larson@umu.se}

\smallskip
\noindent
Karl Larsson, \quad \hfill \addressumushort\\
{\tt karl.larsson@umu.se}

\end{document}